\documentclass[12pt]{amsart}

\usepackage{geometry,amsmath,amssymb,amsthm,bbm, enumerate,mathtools,ytableau}
\usepackage[retainorgcmds]{IEEEtrantools}
\usepackage[all]{xy}

\newtheorem{corollary}{Corollary}
\newtheorem{lemma}{Lemma}
\newtheorem{proposition}{Proposition}
\newtheorem{theorem}{Theorem}

\newtheorem{example}{Example}

\newcommand{\tmtextit}[1]{{\itshape{#1}}}
\newcommand{\C}{{\mathbb{C}}}
\newcommand{\R}{{\mathbb{R}}}

\usepackage{mathdots,mathrsfs}
\DeclareMathAlphabet      {\mathbfit}{OML}{cmm}{b}{it}

\def\lam{{\lambda}}
\def\Lam{{\Lambda}}

\newcommand{\wt}{{\rm wt}}
\newcommand{\diag}{{\rm diag}}
\newcommand{\GL}{{\rm GL}}
\newcommand{\SL}{{\rm SL}}
\newcommand{\Sp}{{\rm Sp}}
\newcommand{\SO}{{\rm SO}}
\newcommand{\GSp}{{\rm GSp}}
\newcommand{\PGSp}{{\rm PGSp}}

\newcommand{\CE}{{\mathcal {E}}}

\newcommand{\FJ}{{\mathfrak J}}

\newcommand{\Fb}{{\mathfrak b}}
\newcommand{\Fc}{{\mathfrak c}}
\newcommand{\Fd}{{\mathfrak d}}

\newcommand{\Fs}{{\mathfrak {s}}}
\newcommand{\Ft}{{\mathfrak t}}
\newcommand{\FD}{{\mathfrak D}}
\newcommand{\FU}{{\mathfrak U}}

\newcommand{\Fx}{{\mathfrak x}}

\newcommand{\Spin}{{\rm Spin}}

\newcommand{\Z}{{\mathbb Z}}

\newcommand{\ud}{\,\mathrm{d}}
\def\bks{{\backslash}}

\newcommand{\cpair}[1]{\left\{{#1}\right)}
\newcommand{\ppair}[1]{\left( {#1} \right)}
\usepackage{tikz}

\begin{document}
 \title{Tokuyama-type formulas for type B}
  \author{Solomon Friedberg and Lei Zhang}
\address{Department of Mathematics, Boston College, Chestnut Hill, MA 02467-3806}
\address{
Department of Mathematics, National University of Singapore\\
Block 17, 10 Lower Kent Ridge Road, Singapore 119076
}
\email{solomon.friedberg@bc.edu}
\email{matzhlei@nus.edu.sg}
\thanks{This work was supported by NSA grant H98230-13-1-0246 (Friedberg).}

\begin{abstract}
We obtain explicit formulas for the product of a deformed Weyl denominator with the character of an irreducible representation of the spin group $\Spin_{2r+1}(\C)$, which is an analogue of the formulas of
Tokuyama for Schur polynomials and Hamel-King for characters of symplectic groups.  To give these, we 
start with a symplectic group and 
obtain such characters using the Casselman-Shalika formula.  We then analyze this using objects
which are naturally attached to the metaplectic double cover of an odd orthogonal group, which also has
dual group $\Spin_{2r+1}(\C)$.
 \end{abstract}
 \maketitle

\section{Introduction}

Our goal in this paper is to establish combinatorial formulas for the product of a deformed Weyl denominator
with the character of an irreducible representation of the spin group $\Spin_{2r+1}(\C)$.
For Cartan type A, the first example of such formula was given by Tokuyama
\cite{Tokuyama},
who expressed a Schur polynomial attached to a dominant weight $\lambda$, multiplied by a deformation of
the Weyl denominator, as a sum over strict Gelfand-Tsetlin 
patterns of weight $\lambda+\rho$, where $\rho$ is the Weyl vector.  This formula may be rephrased in terms
of crystal graphs (see Brubaker, Bump and Friedberg~\cite{BBF11}, Thm.~5.3); in that 
case one attaches a weight to each vertex of the crystal
graph $\mathcal{B}_{\lambda+\rho}$ of type A and highest weight $\lambda+\rho$
and takes a sum over vertices.  It may also be rephrased in terms of other combinatorial
objects such as semistandard Young tableaux.  
Similar formulas for a highest weight character of a symplectic group $\Sp_{2r}(\C)$ were given
by Hamel and King \cite{HK}.  Their basic formula expresses such a character in terms of $sp(2r)$-standard
shifted tableaux; it may be rewritten in terms of other combinatorial objects (see also Ivanov \cite{Iv}).
Our goal here is to present formulas for the dual case, where a symplectic group is replaced by a Spin or odd orthogonal group.
Though dual, this case is fundamentally different; indeed a formula analogous to Hamel and King's had not
previously been conjectured.  As we shall explain, our approach involves the use of metaplectic groups, which
were not required for previous Cartan types but which appear naturally in considering this case.

Our basic formula is given as follows:
For rank $r$, let $D_B(z;t)$ be the deformed Weyl denominator
$$
D_{B}(z;t)=\prod^{r}_{i=1}z^{-\frac{2(r-i)+1}{2}}_{i}\prod^{r}_{i=1}(1+tz_{i})\prod_{1\leq i<j\leq r}(1+tz_{i}z^{-1}_{j})(1+tz_{i}z_{j}).
$$ 
Let $\lambda$ denote a dominant weight for the group $\Spin_{2r+1}(\C)$, and let $\chi_\lambda$ denote the
character of the representation of highest weight $\lambda$.  Let $e_i$ be the cocharacters
described in Section~\ref{sec:preliminaries} below
and let $z=\prod_{i=1}^r e_i(z_i)$ be an element of the split torus.
Then we shall establish formulas for $D_B(z;t)\chi_{\lambda}(z)$. 

To state our main result, we recall that a Gelfand-Tsetlin (or for short, GT) pattern of type C is an array of non-negative integers of the form
$$
P=\begin{matrix}
a_{0,1}&&a_{0,2}&&\cdots&&a_{0,r}&\\
&b_{1,1}&&b_{1,2}&\cdots&b_{1,r-1}&&b_{1,r}\\
&&a_{1,2}&&\cdots&&a_{1,r}&\\
&&&\ddots&&\ddots&&\vdots\\
&&&&&&a_{r-1,r}&\\
&&&&&&&b_{r,r}	
\end{matrix}
$$
such that the entries interleave, that is, such that
$$\min(a_{i-1,j},a_{i,j})\geq b_{i,j}\geq \max(a_{i-1,j+1},a_{i,j+1})$$ for all $i,j$ (omitting any entries not defined),
or, equivalently, $$\min(b_{i+1,j-1},b_{i,j-1})\geq a_{i,j}\geq \max(b_{i+1,j},b_{i,j}).$$
The set of patterns with fixed top row may be identified with a basis for the highest weight representation of the
symplectic group $\Sp_{2r}(\C)$ whose 
highest weight is determined by the top row (Gelfand and Tsetlin \cite{GCet}).
These patterns arise from branching rules; see the discussion in Proctor~\cite{Proctor}, 
which builds on Zhelobenko~\cite{Zhelobenko}.  
They may also be realized in terms of crystal graphs, as in Littelmann~\cite{Littelmann}.

An entry $a_{i,j}$ in a GT-pattern $P$ is called {\it maximal}  if $a_{i,j}=b_{i,j}$ and {\it minimal} if $a_{i,j}=b_{i,j-1}$.
An entry $b_{i,j}$ is called {\it maximal} if $b_{i,j}=a_{i-1,j}$ and {\it minimal} if $1\leq j<r$ and $b_{i,j}=a_{i-1,j+1}$; $b_{i,r}$
is called minimal when it is zero.
If none of these conditions holds, we say that the entry is {\it generic}. 
Note that the definitions of maximality and minimality are opposite to Definition 7 in
Beineke, Brubaker and Frechette \cite{BeBrFr}.
Let $gen(P)$ (resp.~$max(P)$) be the number of generic (resp.~maximal) entries in $P$. 
Define $max_\iota(P)$ for $\iota\in\{0,1\}$ to be the number of maximal entries $x$ in $P$ such that 
$\Fc(x)\equiv \iota\bmod{2}$, 
where $\Fc(a_{i,j})=\sum^{j-1}_{m=i}(b_{i,m}-a_{i,m+1})$ 
and $\Fc(b_{i,j})=\Fc(a_{i,j})+\sum^r_{k=j+1}(a_{i-1,k}+a_{i,k})$.

For example, the pattern 
\begin{equation}\label{pattern}
P=\begin{matrix}
11& & 9& & 7& & 3& & 1& \\
& \bar{11}& & 7& & {\bf 4}& & \bar{3}& & \bar{1}\\
& & {\bf 9}& & {\bf 5}& & \bar{3}& & \bar{1}& \\
& & & {\bf 7}& & \bar{5}& & \bar{3}& & 0\\
& & & & 7& & 5& & {\bf 1}& \\
& & & & & 5& & {\bf 3}& & \bar{1}\\
& & & & & & 5& & 3& \\
& & & & & & & \bar{5}& & {\bf 2}\\
& & & & & & & & {\bf 3}& \\
& & & & & & & & &0 
\end{matrix}
\end{equation}
has generic entries given in boldface and 
the maximal entries overlined.
We have $gen(P)=8$, $max(P)=9$ and $max_1(P)=4$ which counts $b_{1,4}$, $b_{1,5}$, $a_{1,4}$ and $a_{1,5}$.

A GT-pattern is strict if its entries are strictly decreasing across each row.
If $\mu=(\mu_1,\dots,\mu_r)$ is a vector of non-negative integers, let $GT(\mu)$ be the set of strict GT-patterns with top row
\begin{equation}\label{eq:top-row}
(\mu_1+\mu_2+\cdots+\mu_{r-1}+\mu_r,\mu_{2}+\cdots+\mu_{r-1}+\mu_r,\dots,\mu_{r-1}+\mu_r,\mu_r).
\end{equation}
Let $GT^{\circ}(\mu)$  be the subset of all GT-patterns $P$ in $GT(\mu)$ such that $\Fc(x)\equiv 0\bmod{2}$
for all generic entries $x$ in $P$.  (The pattern in \eqref{pattern} is thus in $GT^\circ(2,2,4,2,1)$.) 
See Lemma~\ref{lm:GT} below for a characterization of such patterns.

Let $\upsilon:\Z^r_{\geq0}\to\Z^r_{\geq0}$ be the map 
\begin{equation}\label{upsilon}\upsilon(\mu_1,\dots,\mu_{r-1},\mu_r)=(2\mu_1,\dots,2\mu_{r-1},\mu_r).
\end{equation}
The map $\upsilon$ arises for a conceptual reason involving root systems, as we shall explain in Section~\ref{sec:preliminaries} below.
For $P$ in $GT^{\circ}(\upsilon(\mu))$ note that $max_1(P)$ is even.  For such $P$, define 
$$
G(P)=(-1)^{\frac{max_1(P)}{2}}t^{max(P)-\frac{max_1(P)}{2}}(1+t)^{gen(P)}.
$$
Let 
$$
\wt_i(P)=\sum^r_{k=i-1}a_{i-1,k}-2\sum^r_{k=i}b_{i,k}+\sum^r_{k=i}a_{i,k},
$$
and let $\wt(P)=\sum_{i=1}^r \wt_i(P) e_i^\vee$, where the $e_i^\vee$ are in the weight lattice attached to
$\Spin_{2r+1}(\C)$ as in Section~\ref{sec:preliminaries}.
Thus
$$
z^{-\wt(P)/2}=\prod_{i=1}^r z_i^{-<e_i,\frac{\wt_i(P)}{2}e_i^\vee>}=\prod_{i=1}^r z^{-\frac{\wt_i(P)}{2}}_i.
$$

For example, if $r=1$, then $\upsilon(\mu)=\mu$ and
$$
GT^\circ(\upsilon(\mu))=GT(\mu)=\left\{ \begin{matrix}
	\mu & \\ & b_{1,1}
\end{matrix}~~ \Big| ~~0\leq b_{1,1}\leq \mu \right\}.
$$
For all $P\in GT(\mu)$, $max_1(P)=0$ and
$$
G(P)=\begin{cases}
1  & \text{ if } b_{1,1}=0\\
1+t  & \text{ if } 0<b_{1,1}<\mu\\
t & \text{ if } b_{1,1}=\mu.
\end{cases}
$$
Furthermore, $\wt_1(P)=\mu-2b_{1,1}$.  Note that $\chi_\lambda(z)$ is given by the Schur polynomial $s_\lambda(z^{1/2},z^{-1/2})$.
(Indeed, $\Spin_3$ is isomorphic to $\SL_2$. The torus $e_1(z)$ of $\Spin_3$ is mapped into $\SL_2$ as 
$\diag(z^{1/2},z^{-1/2})$. 
Hence $\chi_\lambda(z)=s_\lambda(z^{1/2},z^{-1/2})$.)

Write $\lambda=\sum_i \lambda_i \epsilon_i$, where $\epsilon_i$ are the fundamental weights for 
$\Spin_{2r+1}(\C)$, so $\lambda+\rho=\sum_i (\lambda_i+1)\epsilon_i$. 
Let $\mu=(\lambda_1+1,\dots,\lambda_r+1)$.  We write $\upsilon(\lambda+\rho):=\upsilon(\mu)$. Then our main result is:
\begin{theorem}\label{main-theorem}
\begin{equation}\label{eq:T-B}
D_B(z;t)\,\chi_{\lambda}(z)=\sum_{P\in GT^{\circ}(\upsilon(\lambda+\rho))}G(P)z^{-\frac{\wt(P)}{2}}.
\end{equation}	
\end{theorem}
\noindent
We also give a formula for the left-hand side in terms of symplectic shifted tableau (Corollary~\ref{tableaux-result}). 
We remark that setting $t=0$ in Theorem~\ref{main-theorem} one may recover the usual Weyl character formula.

We now sketch the proof of this result.
If $G$ is a connected reductive algebraic group defined over a non-archimedean local field $k$, then
Casselman and Shalika \cite{C-S} 
showed that the (normalized) Whittaker function attached to the unramified principal series
for $G_k$ may be expressed as the value of a character on the dual group $G^\vee$.  If one considers
the associated Eisenstein series over a global field $F$, then each global Whittaker coefficient is a 
product of such local Whittaker functions evaluated at suitable powers of the prime attached to $k$.
Accordingly the coefficients of this Eisenstein series capture values of characters.  

One may also consider metaplectic covers of groups, and it turns out that one may establish formulas
for the Whittaker coefficients of metaplectic Eisenstein series in terms of weighted sums over crystal graphs.
Such formulas were first obtained by Brubaker, Bump and Friedberg \cite{BBF5} in the type A case,
while the authors established such a formula for Eisenstein series on covers of odd orthogonal groups in \cite{FZ}.
If the cover has degree $n$, then these formulas involve weighted sums of products of $n$-th order Gauss sums.
When $n=1$, i.e.\ the cover is trivial, the Gauss sums degenerate into elementary functions,
and after applying the result of Casselman-Shalika,
these formulas then turn out to be equivalent to the formulas of Tokuyama and
of Hamel-King.  (See \cite{BBF5}, \cite{BeBrFr}.)

Here we consider the Whittaker coefficients of suitable Eisenstein series on symplectic groups.
Though dual  to \cite{FZ}, this case is fundamentally different.  Indeed, as Brubaker, Bump, Chinta and Gunnells observe \cite{BBCG},
the coefficients attached to 
odd degree covers of odd orthogonal groups and even degree covers of symplectic groups have natural descriptions
(in their work, conjectural) which are modeled on the type A description of  \cite{BBF5}.  
The $1$-fold cover of a symplectic group, which is needed here, is outside this formalism.  That is, the conjectural
formulas for the even covers of symplectic groups do not describe the situation for the odd covers.  Since no such formulas
are known, even conjecturally, we go back to the Eisenstein series and work with them to arrive at and prove
our result.  

Our proof goes as follows. Working with an Eisenstein series on a symplectic group,
we obtain an inductive formula for its Whittaker coefficients.
If $p$ is a prime, then the $p$-parts of these coefficients 
give the desired character at powers of $p$.  We then proceed to 
analyze this inductive expression, which involves an exponential sum.  We rewrite the 
exponential sum in terms of short (Gelfand-Tsetlin) patterns of type B.  These are combinatorial
objects which are related to full Gelfand-Tsetlin
patterns of type B as introduced by Proctor \cite{Proctor} just as short patterns of type A, found in \cite{BBF11}, Ch.~6,
are related to Gelfand-Tsetlin patterns of type A. Our next step is to show that the weighted sum over type B 
short patterns is equal to a weighted sum
over type C short patterns, with the Gauss sums on the type C side corresponding to the double cover (Proposition~\ref{pro:general}).
This turns out to be key.
The type B sum is not easy to work with; indeed its support when translated to the language of crystal graphs
would not be a finite crystal.  However, the type C sum we consider is well behaved.  

The passage to this dual case and double cover is motivated by the observation that 
both $\Sp_{2r}$ and the double cover of $SO_{2r+1}$ have {\sl the same dual group}, where the
dual of such a cover is defined by McNamara \cite{Mc}, Section 13.11.  
The passage from $\mu$ to $\upsilon(\mu)$ is natural in the context of the precise identification.
Then the critical fact is that in working with the 
double cover, the Gauss sums that arise are quadratic Gauss sums modulo $q$ for suitable
$q$.  These once again are evaluable, this time in terms of polynomials in $q^{1/2}$, 
and in a degenerate case they also give $-1$.
The type C sum may then be recast in a more traditional
combinatorial language and the main results established.    

We remark that Tokuyama's formula can be established directly from Pieri's rule; indeed this was the
approach used by Tokuyama. One could similarly establish our results by branching.  However, the 
details (at least as we have carried them out) turn out to be considerably longer than the approach given here.
Also, when one works over a number field with enough roots of unity, 
the calculations we give here may be extended to a metaplectic cover of arbitrary degree of a symplectic group.
However, doing so is considerably more complicated than the results presented here, as
for higher degree covers the Gauss sums are not in general evaluable.
One must use in detail our prior work for covers of type C in carrying out the analysis, which is
quite intricate.  This ultimately leads to a formula which describes the Whittaker coefficients
as sums over a crystal graph which applies to all covers of
symplectic groups. We shall present this in a separate work. We also remark that in type A,
Tokuyama's formula has geometric
significance---each summand represents the contribution
to the Whittaker coefficient from a Mirkovi\'c-Vilonen cycle, as shown
by McNamara \cite{Mc3}.  One could hope for a
similar interpretation of this formula (as well as its generalization to covers).

This paper is organized as follows.
In Section~\ref{sec:preliminaries} the notation is set and the basic Eisenstein series is introduced, and the map $\upsilon$ is recast in terms of dual groups.  In Section ~\ref{section:inductiveformula} the inductive 
formula for the Whittaker coefficients of this Eisenstein series
is established by following the method of the authors \cite{FZ} and Brubaker and Friedberg \cite{Br-Fr}.  
We then restrict to the $p$-power coefficients.
Section~\ref{sec:sumatp} begins our analysis of the inductive expression; short patterns are introduced, adorned with
boxes and circles that keep track of special cases (decorations that may naturally be expressed in terms
of crystal graphs as in \cite{BBF11}, Ch.~2), and the exponential sum is recast as sum over such objects,
with the weighting depending on the decorations.  The relation between the type B and type C sums occupies
Sections~\ref{sec:T-R} and \ref{sec:general}.  This is the most technically demanding part of the paper, as it draws heavily on ideas
from \cite{BBF11} and \cite{FZ}.  Section~\ref{sec:T-R} handles a specific case, the {\sl totally resonant case}, which 
turns out to be the most challenging case, similarly to types A and C. Then in Section \ref{sec:general} the full relation is 
established by reduction to  the totally resonant case. 
In Section~\ref{characters}, the type C expression is used to establish our main result,
Theorem~\ref{main-theorem}.  Then in the last section, the equality is reformulated
in terms of symplectic tableaux, an analogue of the original formula of Tokuyama 
and of the formula of Hamel and King for the dual case.

 \section{Preliminaries}\label{sec:preliminaries}

Let $J_{2r}$ be the skew-symmetric matrix of alternating $1$'s and $-1$'s given by
$$
J_{2r}=\begin{pmatrix}
&&&&-1\\
&&&\iddots&\\
&&1&&\\
&-1&&\\1&&&
\end{pmatrix}.
$$
Define  $\Sp_{2r}$ (or simply $G$) to be the split symplectic group which preserves the symplectic form given by $J_{2r}$, that is,
$$
\Sp_{2r}=\{g\in \GL_{2r}\mid {}^{t}\!gJ_{2r} g= J_{2r}\}.
$$ 
Let $B=TU$ be the Borel subgroup of $G$ consisting of all upper triangular matrices; thus, the torus $T$ consists of elements of the form
$$
d(t_1,\dots,t_r):=\diag(t_{1},t_{2},\dots,t_{r},t^{-1}_{r},\dots,t^{-1}_{2},t^{-1}_{1}).
$$ 

Let  ${\bf s}=(s_{1},s_{2},\dots,s_{r})\in\C^{r}$. 
Define a character on $T_{\R}$ by
$$
\FJ_{\bf s}(d(t_1,\dots,t_r))=\prod^{r}_{i=1}\prod^{i}_{j=1}|t_{j}|^{s_{i}}.
$$
Let $I({\bf s})$ be the space consisting of all smooth complex-valued 
functions $f$ on ${G}_{\R}$ such that 
$$
f(tug)=\FJ_{\bf s}(t)f(g),
$$
for all $t\in{T}_{\R}$, $u\in U_{\R}$, and $g\in {G}_{\R}$.
Then $I({\bf s})$ affords a representation of ${G}_{\R}$ by right translation.

For ${\bf s}$ with $\Re(s_i)$ sufficiently large, $1\leq i\leq r$, and for $f\in I({\bf s})$, define the Eisenstein series
$$
E_{f}(g, {\bf s})=\sum_{\gamma \in B_{\Z}\bks G_{\Z}} f(\gamma g), \qquad g\in {G}_{\R}.
$$
Then this Eisenstein series is absolutely convergent, and
has analytic continuation to all ${\bf s}\in\C^{r}$ and functional equations
(Langlands \cite{Lan}).

To close this Section, let us discuss root systems and clarify the role of the map $\upsilon$ defined in \eqref{upsilon} above.
Our numbering of the simple roots will be consistent with Bourbaki \cite{Bourbaki}, Plates II and III.
Let $\Phi$ be the set of roots of $\PGSp_{2r}$, and $\{\alpha_{i}\}_{1\leq i\leq r}$ with $\alpha_{i}=e_{i}-e_{i+1}$ for $i<r$ and $\alpha_{r}=2e_{r}$
be a set of simple roots.  Then
$\Phi^{\vee}=\{\pm e_{i}^\vee\pm e^\vee_{j},\pm e_{k}^\vee\colon 1\leq i<j\leq r,\ 1\leq k\leq r\}$ is the set of coroots.
The dual group is isomorphic to $\Spin_{2r+1}(\C)$.
The fundamental weights of this dual group are given by
$$
\varepsilon_{i}=\sum^{i}_{j=1}e^\vee_{i} \text{ for }1\leq i<r,\qquad
\varepsilon_{r}=\frac{1}{2}\sum^{r}_{j=1}e^\vee_{i}.
$$
There is also a notion of dual groups associated to an $n$-fold cover of a reductive group $H$ (though these covers 
are not algebraic groups); the roots of the dual
group are the coroots $\alpha^\vee$ of $H$ multiplied by a factor of $n/\gcd(n,\|\alpha^\vee\|^2)$ where $\|\alpha^\vee\|^2$ is the squared length of
the coroot, with the short coroots
normalized to have length 1. See for
example McNamara, \cite{Mc10}, Section 4.3 or \cite{Mc}, Section 13.11.  In particular, 
the root datum of the dual group associated to the double cover of ${\SO}_{2r+1}$ has roots
$\Phi'^\vee=\{2(\pm e'^\vee_i\pm e'^\vee_{j}), \pm 2e'^\vee_k\colon 1\leq i<j\leq r,\ 1\leq k\leq r\}$ .
This dual group is again isomorphic to $\Spin_{2r+1}(\C)$.

Let $\upsilon$ be the linear map given by $\upsilon(e^\vee_{i})=2e'^\vee_{i}$. 
 Then $\upsilon$ identifies $\Phi^\vee$ with $\Phi'^\vee$ and extends to a map of the weight lattices.   
 In particular, on the fundamental weights we have $\upsilon(\varepsilon_{i})=2\varepsilon'_{i}$ for $i<r$,
 $\upsilon(\varepsilon_{r})=\varepsilon'_{r}$.  If $\mu=\sum \mu_i\epsilon_i$, then 
 $\upsilon(\mu)=\sum^{r-1}_{i=1}2\mu_{i}\varepsilon'_{i}+\mu_{r}\varepsilon'_{r}$.  Written in components, this is the map 
 $\upsilon$ appearing in \eqref{upsilon}.
 Note also that $\upsilon^{-1}(\sum^{r}_{i=1}\wt_i(P)e_i'^\vee)=\sum^{r}_{i=1}\frac{\wt_i(P)}{2}e_i^\vee$.
If we let $\wt'(P)=\sum_i \wt_i(P)e'^\vee_i$, then \eqref{eq:T-B} may be rewritten
$$
D_B(z;t)\chi_{\lambda}(z)=\sum_{P\in GT^{\circ}(\upsilon(\lambda+\rho))}G(P)z^{-\upsilon^{-1}(\wt'(P))}.
$$

\section{An inductive formula for the Whittaker Coefficients}\label{section:inductiveformula}

Let $\psi(x)=\exp(2\pi i x)$, and
for each $r$-tuple ${\bf m}=(m_{1},\dots,m_{r})$ of nonzero integers, define a character $\psi_{\bf m}$ on 
the unipotent subgroup $U_{\R}$ by
$$
\psi_{\bf m}(u)=\psi(m_{1}u_{1,2}+m_{2}u_{2,3}+\cdots+m_{r}u_{r,r+1}).
$$
Let $f\in I({\bf s})$. Then the Whittaker coefficients studied in this paper are given by
\begin{equation} 
W_{\mathbf m}(f,\mathbf s)=\int_{U_\Z\backslash U_\R} E_f(u,\mathbf s)\,\psi_{\mathbf m}(u)\ud u.
\end{equation}
For $f\in I({\bf s})$ with $\Re(s_{i})$ sufficiently large and for $t\in T_{\R}$, we also introduce the Whittaker functionals
\begin{align*}
\Lam_{{\bf m},t}(f)&=\FJ(t)^{-1}\int_{U_\R}f(tJ_{2r}u)\,\psi_{\mathbf m}(u)\ud u.
\end{align*}

Our immediate goal is to obtain the following formula for the Whittaker coefficients,
expressing the rank $r$ coefficients in terms of rank $r-1$ coefficients.  Then the bulk of our work below will be
to analyze the combinatorial aspects of the exponential sums that arise in this inductive expression.

\begin{theorem}\label{thm1}   Suppose $\Re(s_{i})$ are sufficiently large. Then the Whittaker coefficient is given by
$$W_{\mathbf m}(f,\mathbf s)=\sum_{C_1,\dots,C_r=1}^\infty
\begin{frac}{H(C_1,\dots,C_r;\mathbf m)}
{C_1^{2s_1}\dots C_r^{2s_r}}\end{frac}
\Lambda_{\mathbf m, \mathbf C}(f),$$
where 
$${\bf C}=d(C^{-1}_{1},C^{-1}_{2}C_{1},\dots,C^{-1}_{r-1}C_{r-2},C^{-1}_{r}C_{r-1}).$$ 
If $r\geq 2$, then the coefficient $H(C_1,\dots,C_r;\mathbf{m})$ satisfies
\begin{align}\label{inductive-coeff}
&H(C_{1},\dots,C_{r};m_{1},\dots,m_{r})\\
=&\sum_{D_{i}, d_{j}=1 }^\infty\,H(D_{2},\cdots,D_{r},m_{2}\frac{d_{1}d_{2r-2}}{d_{2}d_{2r-1}},\dots,m_{r}\frac{d_{r-1}^{2}}{d_{r+1}^{2}})\,\prod^{2r-1}_{k=1}d_{k}^{k}\,L_{k}^{-1}\notag\\
&\times\sum_{\substack{c_{j}\bmod L_{j}\\{(c_j,d_j)=1}}}
\psi\Big(m_{1}\frac{c_{1}}{d_{1}}+m_{2}(\frac{u_{1}c_{2}}{d_{2}}-\frac{d_{1}c_{2r-1}u_{2r-2}}{d_{2}d_{2r-1}})+\cdots
\notag\\&\quad+m_{r-1}(\frac{u_{r-2}{c_{r-1}}}{d_{r-1}}-\frac{d_{r-2}c_{r+2}u_{r+1}}{d_{r-1}d_{r+2}})
+m_r(\frac{u^{2}_{r-1}c_{r}}{d_{r}}+2\frac{d_{r-1}u_{r-1}c_{r+1}}{d_{r+1}d_{r}}+\frac{d^{2}_{r-1}c_{r+1}^{2}u_{r}}{d_{r+1}^{2}d_{r}})\Big).\notag
\end{align}
Here the outer sum is over $D_{i}$ for $2\leq i\leq r$ and $d_{j}$ for $1\leq j\leq 2r-1$ such  that 
$C_{i}=\Fd_{i}D_{i}$ (where $D_{1}=1$ for convenience)
with
$$\Fd_{i}=\prod^{2r-1}_{j=i} d_{j}\prod^{i-1}_{j=1}d_{2r-j}$$
and such that the following divisibility conditions hold:
\begin{alignat*}{2}
&d_{j+1}\vert m_{j+1} d_{j}~&\text{for}~&1\leq j\leq r-2  \\
&d_{j+1}d_{2r-j}\vert m_{j+1} d_{j}d_{2r-j-1}~~~&\text{for}~&1\leq j\leq r-2\\& d^{2}_{r+1}\vert m_{r}d^{2}_{r-1}.&& 
\end{alignat*}
Also, we have set 
$$L_i=
\begin{cases}\prod^i_{j=1}d_j&\mbox{when $1\leq i\leq r-1$}\\
(\prod^{r-1}_{j=1}d_j)^2 d_r&\text{when $i=r$}\\
d^{-1}_{i'-1}(\prod^{i'}_{j=1}d_{r+j})d_r\prod^{r-1}_{j=1}d_j&\text{when $i=r+i'$, $1\leq i'\leq r-1$}
\end{cases}$$
with $d_0=1$, and $u_j$ satisfies $c_ju_j\equiv 1\bmod d_j$.
\end{theorem}

When $d_i$ are powers of a prime $p$, the quantities $\Fd_i$ will turn out to be closely related to the weights
appearing in the crystal graph description.  There is no divisibility condition for $d_{r}$, and indeed
we will see below that $G(\Ft)$ has non-zero value for any $d_{r}$ in the totally resonant case in Section~\ref{sec:T-R}.

\begin{proof}
This is proved by an induction in stages argument,
similar to the calculations in the odd orthogonal case \cite{FZ}.  The study of maximal
parabolic Eisenstein has recently been carried out in generality by Brubaker and Friedberg \cite{Br-Fr}
and the formula given here is a consequence of that work.
However the situation considered here is easier (base field $\mathbb Q$, a specific Lie type, trivial cover),
and one can simplify the argument by appealing to a factorization and
to earlier results for type $A$.  Accordingly we sketch this simplified
proof below for the convenience of the reader.

We have an inductive formula for the Eisenstein series
$$
E_{f}(g,s)=\sum_{\gamma\in P_\Z\bks G_\Z}\sum_{\delta\in B'_\Z\bks G'_\Z}f(\gamma
\delta g),
$$
where $P$ is the standard parabolic subgroup of $G$ with the Levi subgroup $\GL_{1}\times G'$ and $G'\cong\Sp_{2r-2}$.
If $B'$ is the Borel subgroup in $G'$, then we may
identify $B_\Z\bks P_\Z$ with $B'_\Z\bks G'_\Z$.

First, we parametrize the cosets $P_\Z\bks G_\Z$ by using the last row of matrices in $G_\Z$ modulo $\pm1$.
To describe a full set of coset representatives,
consider the three embeddings $i_1:SL_r\to G$, $i_2:SL_2\to G$, $i_3:SL_r\to G$
given as follows.  The embedding $i_1$ is described via blocks:
$$i_1:\begin{pmatrix}g&v\\w&a \end{pmatrix} \hookrightarrow
\begin{pmatrix}a'&&w'&\\&g&&v\\v'&&g'&\\&w&&a \end{pmatrix},
$$
where $g, g'\in M_{(r-1)\times(r-1)}$, $v, v', w^{t}, {w'}^{t}\in M_{(r-1)\times 1}$, and $a, a'\in M_{1\times 1}$;
the primed entries are uniquely determined so that the matrix is in $G$.   The embedding $i_2$ is the embedding of $\SL_2$ into $\Sp_{2r}$:
$$i_2:\begin{pmatrix}a&b\\c&d\end{pmatrix}\hookrightarrow
\begin{pmatrix}
a&&b\\
&I_{2r-2}&\\
c&&d
\end{pmatrix}.
$$
This embedding corresponds to a long root in $\Sp_{2r}$.  
The embedding $i_3$ maps to the Levi subgroup of the Siegel parabolic of $G$:
$$i_3: h\hookrightarrow \begin{pmatrix}h'&\\&h\end{pmatrix}.$$
Let $P_r$ denote the standard parabolic subgroup of $SL_{r}$ of type $(r-1,1)$.
Similarly to Lemmas 2 and 3 in \cite{FZ}, we have the following parametrization in the symplectic case.
\begin{lemma}\label{lm:decomp}
Let ${\bf R}=(R_{1},\dots,R_{2r})\in \Z^{2r}$ and suppose that $\gcd(R_{1},\dots,R_{2r})=1$.
Then there exist $g_{1},g_{3}\in \SL_{r}(\Z)$ and $g_{2}\in \SL_{2}(\Z)$ such that $i_{1}(g_{1})i_{2}(g_{2})i_{3}(g_{3})$ has bottom row ${\bf R}$.
Moreover, a complete set of coset representative for $P_\Z\bks G_\Z$ is given by $i_{1}(g_{1})i_{2}(g_{2})i_{3}(g_{3})$ where $g_{1},g_{3}\in P_{r}(\Z)\bks \SL_{r}(\Z)$ and $g_{2}\in P_{2}(\Z)\bks \SL_{2}(\Z)$. 
\end{lemma}

To compute the Whittaker coefficients, we may restrict to the cosets in the big cell $BJ_{2r}B$
as the others do not contribute.
Let $\gamma$ be a representative of a coset in $P_\Z\bks G_\Z\cap (B_\R J_{2r}B_\R)$, and 
$\gamma'=\gamma J_{2r}$. 
By Lemma~\ref{lm:decomp}, we have $\gamma'=i_1(g_1)\,i_2(g_2)\,i_3(g_3)$ as above.
Moreover, $g_1$ and $g_3$ may themselves be realized as products of $r-1$ embedded $\SL_{2}$'s, as in Brubaker, Bump and Friedberg \cite{BBF5}, pg.\ 1098.  
Decomposing each $\SL_{2}$ via the Bruhat decomposition, we obtain a factorization
$\gamma'=\FU^{+}\FD\FU^{-}$ where $\FU^{+}$ is in the unipotent radical of $P$ and $\FU^{-}$ is a lower triangular matrix. Moreover, 
\begin{align*}
\psi_{{\bf m}}(J^{-1}_{2r}\FU^{-}J_{2r})
=&\psi\left(m_{1}\frac{c_{1}}{d_{1}}+\sum^{r-2}_{i=1}m_{i+1}\ppair{\frac{c_{i+1}u_{i}}{d_{i+1}}-\frac{d_{i}u_{2r-i-1}c_{2r-i}}{d_{i+1}d_{2r-i}}}\right.\\
&\left.+m_{r}\ppair{\frac{u^{2}_{r-1}c_{r}}{d_{r}}+2\frac{d_{r-1}u_{r-1}c_{r+1}}{d_{r+1}d_{r}}+\frac{d^{2}_{r-1}c_{r+1}^{2}u_{r}}{d_{r+1}^{2}d_{r}}}\right),
\end{align*}
where the $(c_i,d_i)$ are the bottom rows of the embedded $SL_2(\Z)$ matrices.
Now one passes to the double cosets $\gamma\in P_\Z\bks G_\Z\cap (B_\R J_{2r}B_\R)/ U_{P}(\Z)$
and carries out an unfolding in the usual way.
It may be checked by considering bottom rows that a full set of double coset
representatives is obtained by taking $c_{i}$ modulo $L_i$ with $L_i$ as in the Theorem.
Following the calculations of \cite{BBF5} and \cite{FZ}, the Theorem follows. 
\end{proof}

\begin{corollary}\label{factorization}  Suppose that $C_i=\prod_p p^{\gamma_{p,i}}$, $m_j=\prod_p p^{\nu_{p,j}}$ are factorizations into distinct prime powers.
Then $H(C_1,\dots,C_r;\mathbf m)=\prod_{i,j} H(p^{\gamma_{p,1}},\dots,p^{\gamma_{p,r}};p^{\nu_{p,1}},\dots,p^{\nu_{p,r}}).$
\end{corollary}

	\begin{proof}  This may be proved directly from \eqref{inductive-coeff} by an induction argument, making use of the Chinese remainder theorem (compare \cite{FZ}, Ch.~5, for the odd orthogonal case).
The basic argument is carried out in greater generality in
 \cite{Br-Fr}, Section 6, where it is proved that exponential sums arising in the Whittaker expansion of a maximal parabolic
metaplectic Eisenstein series on a covering group factor into prime power contributions
at the expense of certain power residue symbols.  By induction this gives a similar ``twisted multiplicativity" factorization
of $H(C_1,\dots,C_r;\mathbf m)$.  
In our case, the cover is trivial and hence the residues symbols are all identically $1$.  Thus $H$ factors into its prime power contributions.  
\end{proof}

On the other hand, after passing to the adeles using strong approximation, 
the global Whittaker coefficient $W_{\mathbf m}(f,\mathbf s)$ factors into local Whittaker coefficients given by $p$-adic integrals. 
In view of Corollary~\ref{factorization}, this local Whittaker coefficient may be expressed as 
the sum of the coefficients $H(p^{\gamma_{p,1}},\dots,p^{\gamma_{p,r}};p^{\nu_{p,1}},\dots,p^{\nu_{p,r}})$ (times powers of $p^{\gamma_{p,i}s_i}$), the sum
over $\gamma_{p,i}$.   But by the Casselman-Shalika formula \cite{C-S}, at almost all places $p$ this local Whittaker coefficient
is also a value of a character attached to an irreducible representation of a group of type B, whose highest weight depends on the $\nu_{p,j}$'s.
This will be the key to giving new combinatorial realizations of characters, but we need to say more.
 
Let $H$ be any connected reductive group defined over a non-archimedean local field $F$.  Let $B=TU$ be a Borel subgroup of $H$,
where $T$ is a maximal split torus and $U$ is the unipotent radical of $B$.
An unramified principal series for $H_F$
is a certain space of functions on $H_F$ (defined by its behavior under the left action of $B_F$, under which it transforms by a fixed character of $T_F$)
equipped with a right $H_F$-action.  Then the local Whittaker function is given by
$$W(h)=\int_{U^-_{F}} f(u h) \, \psi(u)\,du,$$
where $U^-$ is the opposite unipotent subgroup, $\psi$ is a smooth complex character of $U^-_{F}$,
and $f$ is the spherical vector.  (This converges for principal series attached to certain characters of $T_F$ and may be extended by analytic
continuation to the remaining cases.) If $h\in T_F$, the variable change $u\mapsto huh^{-1}$ allows one to remove $h$
at the expense of acting on the character $\psi$.  If we were to apply this with $H=\Sp_{2r}$, the action of $T_F$
would give only half the characters we consider, as the $\nu_{p,r}$ above would only change by even integers.
Thus we must apply the Casselman-Shalika formula on the group of symplectic similitudes $\GSp_{2r}$ (or on its quotient by the center)
in order to capture all the coefficients computed above.
Then the Casselman-Shalika formula expresses the Whittaker coefficients considered here
as values of characters of the irreducible representation $\chi_\lambda$ of the group $\Spin_{2r+1}(\C)$ of highest weight $\lambda$.  
These must be multiplied by a factor that accounts for the difference between normalized and unnormalized Whittaker functionals,
given by a deformed Weyl denominator.
Thus we obtain the following relation.  

Let $\lambda=\sum_{i=1}^r \lambda_i\epsilon_i$ be a dominant weight; hence
$\lambda+\rho=\sum_{i=1}^r \mu_i\epsilon_i$ with $\mu_i=\lambda_i+1$.
Let  $a_{0,i}=\mu_r+2\sum_{j=i}^{r-1}\mu_j$.  (The $a_{0,i}$ will be realized as
the top row of a set of GT patterns in Section~\ref{characters} below.)
Write $p^\lam=(p^{\lambda_1},\dots,p^{\lambda_r})$. 
\begin{proposition}\label{first-formula}  Let $q$ be a prime, and suppose that $z_iz^{-1}_{i+1}=q^{1-2s_i}$ for $1\leq i\leq r-1$ and $z_r=q^{1-2s_r}$. 
Then
$$
D_B(z;-q^{-1})\chi_{\lam}(z)
=\sum_{\bf k}H^\flat(p^{\bf k};p^{\lam})z_1^{k_1-\frac{a_{0,1}}{2}}z_2^{k_2-k_1-\frac{a_{0,2}}{2}}
\cdots z_r^{k_{r}-k_{r-1}-\frac{a_{0,r}}{2}}
$$
where 
 $H^{\flat}(p^{\bf k};p^{\lam})=q^{-\sum_{i=1}^r k_i} H(p^{\bf k};p^{\lam})$.
Here the sum is over ${\bf k}=(k_1,\dots,k_r)$ in $\Z^r_{\geq0}$,
and $p^{\bf k}:=(p^{k_1},\dots,p^{k_r})$.
\end{proposition}

\section{The Sum at $p$: $G(\Ft)$}\label{sec:sumatp}

Our goal is to describe the 
coefficients $H$ at powers of $p$.  
Below the letters $m_{i}$, $d_{i}$, $L_i$  denote non-negative integers.  (They  may be regarded
as the orders at $p$ of the corresponding integers in the prior section.)  
In particular
$$L_i=
\begin{cases}\sum^i_{j=1}d_j&\mbox{when $1\leq i\leq r-1$}\\
d_r+2\sum^{r-1}_{j=1}d_j &\text{when $i=r$}\\
d_r-d_{i'-1}+\sum^{i'}_{j=1}d_{r+j}+\sum^{r-1}_{j=1}d_j&\text{when $i=r+i'$, $1\leq i'\leq r-1$}
\end{cases}$$
with  $d_0=0$.
Broadly speaking, we apply the approach in \cite{BBF11, FZ}, and use language comparable to those cases.
Let $\Ft=(d_{1},d_{2},\dots,d_{2r-1})$.

We will describe the support and give a preliminary combinatorial version of the following sum. 
Let
\begin{multline*}
G(\Ft)=\prod^{2r-1}_{k=1}p^{-L_k}\cdot\hspace{-1em}\sum_{\substack{c_{j}\bmod p^{L_{j}}, 1\leq j\leq 2r-1\\{(c_j,p)=1 \text{~if $d_j>0$}}}}\!\!
\psi\left(p^{m_{1}}\frac{c_{1}}{p^{d_{1}}}+p^{m_{2}}(\frac{u_{1}c_{2}}{p^{d_{2}}}-\frac{p^{d_{1}}c_{2r-1}u_{2r-2}}{p^{d_{2}+d_{2r-1}}})+\cdots\right.\\
\left.+p^{m_{r-1}}(\frac{u_{r-2}{c_{r-1}}}{p^{d_{r-1}}}-\frac{p^{d_{r-2}}c_{r+2}u_{r+1}}{p^{d_{r-1}+d_{r+2}}})
+p^{m_{r}}(\frac{u^{2}_{r-1}c_{r}}{p^{d_{r}}}+2\frac{p^{d_{r-1}}u_{r-1}c_{r+1}}{p^{d_{r+1}+d_{r}}}+\frac{p^{2d_{r-1}}c_{r+1}^{2}u_{r}}{p^{2d_{r+1}+d_{r}}})\right).
\end{multline*}
Here we suppose that
\begin{align*}
d_{j+1}&\leq m_{j+1} + d_{j}~\text{for}~1\leq j\leq r-2  \\
d_{j+1}+d_{2r-j}&\leq m_{j+1} +d_{j}+ d_{2r-j-1}~\text{for}~1\leq j\leq r-2,\,\,\ 2d_{r+1}\leq m_{r}+2d_{r-1},
\end{align*}
and we set $u_j=0$ if $d_j=0$.

Let $\mu$ be the $r$-tuple of positive integers given by $\mu_{i}=m_{i}+1$ for $1\leq i\leq r$.
Define $CQ^{\leq}_{1}(\mu)$ (the first set of combinatorial quantities attached to $\mu$)
to be the set of $(2r-1)$-tuples $\Ft=(d_1,\dots,d_{2r-1})$ of non-negative integers satisfying the inequalities
\begin{equation}\label{ineq:CQ-d}
\begin{cases}
d_{i}\leq \mu_{i}& 1\leq i\leq r,\\
d_{r}\leq \mu_{r}+2(d_{r-1}-d_{r+1})&\\
d_{2r-i}\leq \mu_{i+1}+d_{i}-d_{i+1}& 1\leq i\leq r-2.
\end{cases}
\end{equation} 
Let $CQ_{1}(\mu)$ be the set of $(2r-1)$-tuple of non-negative integers satisfying~\eqref{ineq:CQ-d} except possibly for the middle inequality for $d_{r}$. 
By analogy with \cite{BBF11} and \cite{FZ}, we call $\Ft$ a {\it short pattern}.

Let $\Ft$ be in $CQ_{1}(\mu)$.
The {\it weighting vector} $k(\Ft)=(k_{1}(\Ft),\dots,k_{r}(\Ft))$ is defined by
$$
k_{i}(\Ft)=\sum^{2r-1}_{j=i}d_{j}+\sum^{i-1}_{j=1}d_{2r-j} \text{ for } 1\leq i\leq r.
$$
When the choice of $\Ft$ is clear, we write $k_{i}$ in place of $k_{i}(\Ft)$.
Let
$$
\Delta(\Ft)=(\Fc_{1},\dots,{\Fc}_{2r-1})
$$
with
$
\Fc_{i}=\sum^{2r-1}_{j=i}d_{i} \text{ for } 1\leq i\leq 2r-1.
$
Thus, $\Fc_{i}-\Fc_{i+1}=d_{i} \text{ for } 1\leq i\leq 2r-1$ and
$\Fc_{i}+{\Fc}_{2r-i+1}=k_{i}$ for $1\leq i\leq r$ (with $\Fc_{2r}=0$ for convenience).  
A short pattern $\Ft$ is called {\it maximal} in $CQ_{1}(\mu)$ if $d_i=d_{2r-i}=\mu_i$ for $1\leq i\leq r$.

We will attach algebraic quantities to the entries of $\Ft$.  This is analogous to the
association of number theoretic quantities to string data in 
 \cite{BBF5}, \cite{FZ}, with the proviso that Gauss sums are all evaluable in this case.  
To keep track of the different possibilities that turn out to arise, as in \cite{BBF5}, \cite{FZ}
we `decorate,' or add circles or boxes, to the entries of $\Delta(\Ft)$ as follows:
\begin{enumerate}
\item  The entry $\Fc_{i}$  is {\it circled} if $d_{i}=0$.
\item The entry $\Fc_{i}$ is {\it boxed} if an equality holds in the upper-bound inequalities~\eqref{ineq:CQ-d}, as follows.
When $1\leq i\leq r$ the entry $\Fc_{i}$ is boxed if $d_{i}=\mu_{i}$;
the entry ${\Fc}_{r+1}$ is boxed if $d_{r}=\mu_{r}+2(d_{r-1}-d_{r+1})$;
when $1\leq i\leq r-2$ the entry ${\Fc}_{2r-i}$ is boxed if $d_{2r-i}=\mu_{i+1}+d_{i}-d_{i+1}$.
\end{enumerate}
For the entry $a=\Fc_{i}$ ($1\leq i\leq 2r-1$), set
$$
\gamma(a)=\begin{cases}
1-q^{-1}& \text{ if $a$ is neither boxed nor circled,}\\
-q^{-1}& \text{ if $a$ is boxed but not circled,}\\
1& \text{ if $a$ is circled but not boxed,}\\
0& \text{ if $a$ is both boxed and circled.}
\end{cases}
$$
Define $G_{\Delta}(\Ft)=\prod^{2r-1}_{i=1}\gamma(\Fc_{i})$.
It is easy to check that that $G_{\Delta}(\Ft)$ is supported in $CQ^{\leq}_{1}(\mu)$.

We now have the following evaluation of the exponential sum $G(\Ft)$ in terms of $G_\Delta(\Ft)$.

\begin{proposition}\label{pro:G}
\begin{enumerate}
\item 
If $\Ft$ is not in $CQ_{1}(\mu)$, then $G(\Ft)=0$.

\item If $\Ft$ is in $CQ_{1}^{\leq}(\mu)$, then $G(\Ft)=G_{\Delta}(\Ft)$.

\item If $d_{r-1}=d_{r+1}$ and $d_{r}-\mu_{r}$ is a positive odd integer, then
$$
G(\Ft)=(1-q^{-1})q^{-\frac{(d_{r}-\mu_{r}+1)}{2}}\prod_{\substack{1\leq i\leq 2r-1\\ i\ne r, r+1}}\gamma(\Fc_{i}),
$$
where $\Fc_{i}$ is the decorated entry in $\Delta(\Ft)$;
if $d_{r-1}=d_{r+1}$ and $d_{r}-\mu_{r}$ is a positive even integer, then, $G(\Ft)=0$.
\end{enumerate}

\end{proposition}

\begin{proof}
This follows from a long but elementary explicit calculation of the exponential sums. 
First, we evaluate the sum over $c_{i}\bmod p^{L_{i}}$ for $r+2\leq i\leq 2r-1$. 
Similar to the calculations in \cite{BBF5} and \cite{FZ}, it is straightforward to obtain the term $\prod^{2r-1}_{i=r+2}\gamma(\Fc_{i})$.

Next, we compute the terms containing the long root of length 2.  Let
$$
G_{2}(\Ft):=q^{-L_{r}-L_{r+1}}\sideset{}{'}\sum_{\substack{c_{r}\bmod p^{L_{r}}\\ c_{r+1}\bmod p^{L_{r+1}}}}
\psi\left( p^{m_{r}}\ppair{\frac{u^{2}_{r-1}c_{r}}{p^{d_{r}}}+2\frac{p^{d_{r-1}}u_{r-1}c_{r+1}}{p^{d_{r}+d_{r+1}}}+\frac{p^{2d_{r-1}}c_{r+1}^{2}u_{r}}{p^{d_r+2d_{r+1}}}}\right)
$$
where the prime indicates that $(c_r,p)=1$ if $d_r>0$ and that $(c_{r+1},p)=1$ if $d_{r+1}>0$.
We suppress the possible dependence of this sum on $u_{r-1}$ from the notation.  The point is that since this sum involves squares, we must handle it
differently.  

If $d_{r-1}=d_{r+1}=0$ and $d_{r}\ne 0$, then  we may set $u_{r-1}=u_{r+1}=0$.  Thus 
\begin{align*}
q^{L_{r}+L_{r+1}}G_{2}(\Ft)=&  \sum_{\substack{c_{r}\bmod p^{L_r}\\(c_{r},p)=1}}
\,\sum_{ \ c_{r+1}\bmod p^{L_{r+1}}} \psi(c^{2}_{r+1}u_{r}p^{m_{r}-d_{r}})\\
=&\sum^{L_{r+1}-1}_{m=0}\sum_{\substack{c'_{r+1}\bmod p^{L_{r+1}-m}\\(c'_{r+1},p)=1}}\sum_{\substack{c_{r}\bmod p^{L_{r}}\\(c_{r},p)=1}} \psi(c'^{2}_{r+1}u_{r}p^{2m+m_{r}-d_{r}})
+\sum_{\substack{c_{r}\bmod p^{L_{r}}\\(c_{r},p)=1}} 1.
\end{align*}
These sums are easily evaluable. When $d_{r}>\mu_{r}$ and $d_{r}\not\equiv \mu_{r}\bmod{2}$, we get
\begin{align*}
q^{L_{r}+L_{r+1}}G_{2}(\Ft)=&\sum^{L_{r+1}-1}_{m=\frac{d_{r}-\mu_{r}+1}{2}}(1-q^{-1})q^{L_{r+1}-m}\cdot q^{L_{r}}(1-q^{-1})+q^{L_{r}}(1-q^{-1})\\
=&q^{L_{r}+L_{r+1}-\frac{(d_{r}-\mu_{r}+1)}{2}}(1-q^{-1}).
\end{align*}
Similarly, when $d_{r}\geq \mu_{r}$ and $d_{r}\equiv \mu_{r}\bmod{2}$, we get instead
\begin{align*} 
&q^{L_{r}+L_{r+1}}G_{2}(\Ft)\\
=&q^{L_{r+1}}(1-q^{-1})\cdot\Big(q^{-\frac{d_{r}-\mu_{r}}{2}}q^{L_{r}}(-q^{-1})+
\sum^{L_{r+1}-1}_{m=\frac{d_{r}-\mu_{r}}{2}+1} q^{-m}\cdot q^{L_{r}}(1-q^{-1})\Big)\\
&\qquad
+q^{L_{r}}(1-q^{-1})\\
=&0.
\end{align*}
Finally, when $0\leq d_{r}<\mu_{r}$, it is easy to verify that $G_{2}(\Ft)= \gamma(\Fc_{r})\gamma(\Fc_{r+1})$.

Next, suppose that $d_{r-1}=d_{r+1}\ne 0$ and $d_{r}\ne 0$. 
After changing $c_{r}$ to $c_{r-1}c_{r}c_{r+1}$, we have
$$
G_{2}(\Ft)=q^{-L_{r}-L_{r+1}} \sum_{\substack{c_{r}\bmod p^{L_{r}}\\ (c_{r},p)=1}}  \sum_{\substack{c_{r+1}\bmod p^{L_{r+1}}\\ (c_{r+1},p)=1}}
\psi\left( \frac{p^{m_{r}}u_{r-1}u_{r}c_{r+1}}{p^{d_{r}}}(c_{r}+1)^{2}\right).
$$
In order to calculate $G_{2}(\Ft)$, we split the sum over $c_{r}$ into the following three pieces: (i) $(c_{r}+1,p^{L_{r}})=1$; (ii) $(c_{r}+1,p^{L_{r}})=p^{m}$ for $0<m<L_{r}$; (iii) $(c_{r}+1,p^{L_{r}})=p^{L_{r}}$. 
Denote the summation over each piece by $G_{2,i}(\Ft)$, $G_{2,ii}(\Ft)$ and $G_{2,iii}(\Ft)$ respectively.
First, 
$$
G_{2,i}(\Ft)=q^{-L_{r}}\gamma(\Fc_{r+1}) \sum_{\substack{c_{r}\bmod p^{L_{r}}\\ (c_{r},p^{L_{r}})=(c_{r}+1,p^{L_{r}})=1}}1 
=\gamma(\Fc_{r+1})(1-2q^{-1}).
$$
Second,  for $(c_{r}+1,p^{L_{r}})=p^{m}$ with $0<m<L_{r}$, write $c'_{r}=p^{-m}(c_{r}+1)$.  Then
$$
G_{2,ii}(\Ft)=q^{-L_{r}-L_{r+1}}\sum^{L_{r}-1}_{m=1}
\sum_{\substack{c'_{r}\bmod p^{L_{r}-m}\\ (c'_{r},p^{L_{r}-m})=1}}
\sum_{\substack{c_{r+1}\bmod p^{L_{r+1}}\\ (c_{r+1},p^{L_{r+1}})=1}}
\psi\ppair{\frac{p^{m_{r}+2m}u_{r-1}u_{r}c_{r+1}c'^{2}_{r}}{p^{d_{r}}}}.
$$
Similarly to the above, we consider separately the three cases $d_{r}>\mu_{r}$ and $d_{r}\not\equiv \mu_{r}\bmod{2}$; 
$d_{r}> \mu_{r}$ and $d_{r}\equiv \mu_{r}\bmod{2}$; $d_{r}\leq \mu_{r}$. After a short calculation, one finds that 
$$
G_{2,ii}(\Ft)=\begin{cases}
(1-q^{-1})(q^{}-q^{-L_{r}})& \text{ if $d_{r}>\mu_{r}$ and $d_{r}\not\equiv \mu_{r}\bmod{2}$} \\
-q^{-L_{r}}(1-q^{-1})& \text{ if $d_{r}> \mu_{r}$ and $d_{r}\equiv \mu_{r}\bmod{2}$}\\
(1-q^{-1})(q^{-\frac{(d_{r}-\mu_{r}+1)}{2}}-q^{-L_{r}})& \text{ if $d_{r}\leq \mu_{r}$}.
\end{cases}
$$
Third, for the piece $(c_{r}+1,p^{L_{r}})=p^{L_{r}}$, 
$$
G_{2,iii}(\Ft)=q^{-L_{r}-L_{r+1}}\sum_{\substack{c_{r+1}\bmod p^{L_{r+1}}\\ (c_{r+1},p^{L_{r+1}})=1}}1=q^{-L_{r}}(1-q^{-1}).
$$
Combining the formulas, we obtain
\begin{align*}
G_{2}(\Ft)=&G_{2,i}(\Ft)+G_{2,ii}(\Ft)+G_{2,iii}(\Ft)\\
=&\begin{cases}
\gamma(\Fc_{r})\gamma(\Fc_{r+1})& \text{ if $d_{r}\leq \mu_{r}$}\\
0& \text{ if $d_{r}> \mu_{r}$ and $d_{r}\equiv \mu_{r}\bmod{2}$}\\
(1-q^{-1})q^{-\frac{d_{r}-\mu_{r}+1}{2}} & \text{ if $d_{r}>\mu_{r}$ and $d_{r}\ne \mu_{r}\bmod{2}$.}
\end{cases}
\end{align*}

In the remaining cases, the evaluation of $G_{2}(\Ft)$ is straighforward.
We then compute the remaining summands over $c_i$ with $i<r$ by an elementary computation.  The Proposition follows.
\end{proof}

\section{The Totally Resonant Case} \label{sec:T-R}
In this section, we will consider  {\it  totally resonant} short patterns $\Ft$.  By definition, these are the ones
satisfying $d_{i}=d_{2r-i}$ for all $i$, or equivalently $k_{i}=k_{i+1}$ for $1\leq i\leq r-1$.  When $\Ft$ is totally
resonant we only record $(d_1,\dots,d_r)$ since the remaining $d_j$ are determined from this data.
A totally resonant short pattern $\Ft$ in $CQ_{1}(\mu)$ is maximal if $\Ft=\mu$. 

For  $1\leq i\leq r$, define the sets
$$
\Omega^{\land}_{i}(\mu)=\{(d_{1},\dots,d_{r})\in \Z^{r}_{\geq 0}\mid d_{j}\leq \mu_{j}\text{ for }j\leq i-1,
d_{i}\land\mu_{i} \text{ and } d_{j}=\mu_{j} \text{ for }j\geq i+1\},
$$
 where $\land$ may be replaced by $<$, $=$, $\geq$ and so on. 
Let $\Omega_{i}(\mu)=\Omega^{<}_{i}(\mu)\cup\Omega^{\geq}_{i}(\mu)$.
For each $\Ft=(d_{1},\dots,d_{r})$ in $\Omega^{\land}(\mu)$, define the {\it weightings} of Type A and Type B, 
$k_{A}(\Ft)=\sum^{r}_{i=1}d_{i}$ and $k_{B}(\Ft)=d_{r}+2\sum^{r-1}_{i=1}d_{i}$.
Let $\Omega^{\land}_{i,R}(\mu,k)$ be the subset of $\Omega^{\land}(\mu)$ consisting of elements $\Ft$ such that $k_{R}(\Ft)=k$ where $R$ is $A$ or $B$.
For convenience we write $\Omega^{
 \land}(\mu):=\Omega^{\land}_{r}(\mu)$ and $\Omega^{\land}_{R}(\mu,k):=\Omega^{\land}_{r,R}(\mu,k)$.
For each non-maximal totally resonant short pattern $\Ft=(d_{1},d_{2},\dots,d_{r})\in \Omega(\mu)$,
let $i_{\Box}(\Ft)$ or simply $i_{\Box}$ denote the quantity $\max\{i\mid 1\leq i\leq r, d_{i}< \mu_{i}\}$.
Define $\Omega(\Ft)$ to be  the subset of $\Omega^{\leq}(\mu)$ 
such that for each $(x_{1},\dots,x_{r})\in \Omega(\Ft)$,
$x_{i}\leq d_{i}$ for $i_{\Box}(\Ft)\leq i\leq r$ and $x_{i}=d_{i}$ for $1\leq i< i_{\Box}(\Ft)$. 
If $\Ft$ is maximal, let $\Omega(\Ft)=\Omega^{\leq}(\mu)$.
Also, if $\Fs=(b_{1},\dots,b_{r})$ is a short pattern, 
let $\Fs_{i}=(b_{1},\dots,b_{i})$ denote the first $i$ components of  $\Fs$.

In order to handle the contribution outside $CQ_{1}^{\leq}(\mu)$, we introduce the following lemma.
\begin{lemma}\label{pro:Omega-G-A}
Let $\Fs=(b_{1},b_{2},\dots,b_{r})$ be a short pattern in $\Omega^{\leq}(\mu)$. 
\begin{enumerate}
\item
If $\Fs$ is maximal, then $\sum_{\Ft\in\Omega(\Fs)}G_{\Delta}(\Ft)q^{k_{A}(\Ft)}=0$.
\item
If $\Fs$ is not maximal, then
\begin{align}\label{eq:Omega-G-A}
\sum_{\Ft\in\Omega(\Fs)}G_{\Delta}(\Ft)q^{k_{A}(\Ft)}
=& q^{k_{A}(\Fs)-(r-i_{\Box})}G_{\Delta}(\Fs_{i_{\Box}})\begin{cases}
(1-q^{-1})^{-1}  & \text{ if $b_{i_{\Box}}>0$,}\\
1 & \text{ if $b_{i_{\Box}}=0$}.
\end{cases}
\end{align}
Here $i_{\Box}=i_{\Box}(\Fs)$ and $\Fs_{i_{\Box}}$ is in $\Omega^{\leq}((\mu_{1},\dots,\mu_{i_{\Box}}))$.
In particular, if $r=1$, then
$$
\sum_{\Ft\in\Omega(\Fs)}G_{\Delta}(\Ft)q^{k_{A}(\Ft)}=\begin{cases}
0 & \text{ if $\Fs$ is maximal,}\\
q^{k_{A}(\Fs)}   & \text{ otherwise.}
\end{cases}
$$
\end{enumerate}
\end{lemma}

\begin{proof}
If $r=1$, then $\Fs=b_{1}=k_{A}(\Fs)$, $\Omega(\Fs)=\{d_{1}\mid 0\leq d_{1}\leq k_{A}(\Fs)\}$ and $G_{\Delta}(\Ft)=\gamma_{A}(\Fc_{1})$. If $b_{1}<\mu_{1}$,  
$$
\sum_{\Ft\in\Omega(\Fs)}G_{\Delta }(\Ft)q^{k_{A}(\Ft)}=1+\sum^{b_{1}}_{d_{1}=1}(1-q^{-1})q^{d_{1}}=q^{b_{1}}.
$$
If $b_{1}=\mu_{1}$, then 
$$
\sum_{\Ft\in\Omega(\Fs)}G_{\Delta }(\Ft)q^{k_{A}(\Ft)}=1+\sum^{\mu_{1}-1}_{d_{1}=1}(1-q^{-1})q^{d_{1}}+(-q^{-1})q^{\mu_{1}}=
q^{\mu_{1}-1}+(-q^{-1})q^{\mu_{1}}=0.
$$

Suppose that $r\geq 2$.
We have two cases, $i_{\Box}=r$ and $i_{\Box}<r$. 
Let $\Ft=(d_{1},\dots,d_{r})$ be in $\Omega(\Fs)$.

Suppose $i_{\Box}=r$.
 Then $d_{i}=b_{i}$ for $1\leq i\leq r-1$ and $d_{r}\leq b_{r}$.
 Recall that
$$
G_{\Delta}(\Ft)=G_{\Delta}(\Ft_{r-1})\gamma(\Fc_{r})\gamma(\Fc_{r+1}),
$$
and  then
$$
\sum_{\Ft\in\Omega(\Fs)}G_{\Delta}(\Ft)q^{k_{A}(\Ft)}
=\sum^{b_{r}}_{d_{r}=0}G_{\Delta}(\Ft)q^{k_{A}(\Ft)}
=G_{\Delta}(\Fs_{r-1})q^{k_{A}(\Fs_{r-1})}\sum^{b_{r}}_{d_{r}=0}
\gamma(\Fc_{r}) \gamma(\Fc_{r+1})q^{d_{r}}.
$$
Since $d_{r}\leq b_{r}<\mu_{r}$, the entries $\Fc_{r}$ and  $\Fc_{r+1}$ are not boxed. 
When $b_{r}=0$, it is easy to see that $\Omega(\Fs)=\{\Fs\}$ and Eqn.~\eqref{eq:Omega-G-A} holds. 
Let us assume that $b_{r}>0$.
If $b_{r-1}=0$, then ${\Fc}_{r+1}$ is circled, $\gamma({\Fc}_{r+1})=1$
and
$$
\sum^{b_{r}}_{d_{r}=0}\gamma(\Fc_{r}) \gamma({\Fc}_{r+1})q^{d_{r}}=1+\sum^{b_{r}}_{d_{r}=1}(1-q^{-1})q^{d_{r}}=q^{b_{r}}.
$$
In this case, $\gamma(\Fs_{r}) \gamma({\Fs}_{r+1})=1-q^{-1}$.
If $0<b_{r-1}\leq \mu_{r-1}$, 
then ${\Fc}_{r+1}$ is neither circled nor boxed, $\gamma({\Fc}_{r+1})=1-q^{-1}$, and 
$$
\sum^{b_{r}}_{d_{r}=0}\gamma(\Fc_{r}) \gamma({\Fc}_{r+1})q^{d_{r}}=1\cdot (1-q^{-1})+\sum^{b_{r}}_{d_{r}=1}(1-q^{-1})^{2}q^{d_{r}}
=(1-q^{-1})q^{b_{r}}.
$$
In this case, $\gamma(\Fs_{r}) \gamma({\Fs}_{r+1})=(1-q^{-1})^{2}$. 
In summary, if $i_{\Box}=r$, then
\begin{equation}\label{eq:Omega-pf-unbox}
\sum^{b_{r}}_{d_{r}=0}G_{\Delta}(\Ft)q^{k_{A}(\Ft)}=\begin{cases}
q^{k_{A}(\Fs)}(1-q^{-1})^{-1}G_{\Delta}(\Fs) &\text{ if $b_{r}>0$;}\\
q^{k_{A}(\Fs)}G_{\Delta}(\Fs) &\text{ if $b_{r}=0$.}
\end{cases}
\end{equation}

Suppose instead that $i_{\Box}<r$. Then  $b_{r}=\mu_{r}$ and the entry $\Fs_{r}$ is boxed but not circled.
First,  we compute
$$
\sum^{\mu_{r}}_{d_{r}=0}G_{\Delta}(\Ft)q^{k_{A}(\Ft)}=G_{\Delta}(\Ft_{r-1})q^{k_{A}(\Ft_{r-1})}\sum^{\mu_{r}}_{d_{r}=0}
\gamma(\Fc_{r}) \gamma({\Fc}_{r+1})q^{d_{r}}.
$$
If $d_{r-1}=0$, then ${\Fc}_{r+2}$ is not boxed and 
$$
\sum^{\mu_{r}}_{d_{r}=0}\gamma(\Fc_{r}) \gamma({\Fc}_{r+1})q^{d_{r}}=1+\sum^{\mu_{r}-1}_{d_{r}=1}(1-q^{-1})q^{d_{r}}+0=q^{\mu_{r}-1}.
$$
If $0<d_{r-1}\leq \mu_{r-1}$, then
$$
\sum^{\mu_{r}}_{d_{r}=0}\gamma(\Fc_{r}) \gamma({\Fc}_{r+1})q^{d_{r}}
=(1-q^{-1})+\sum^{\mu_{r}-1}_{d_{r}=1}(1-q^{-1})^{2}q^{d_{r}}+q^{-2}q^{\mu_{r}}=q^{\mu_{r}-1}.
$$
Thus we obtain an inductive formula
\begin{equation}
\sum^{\mu_{r}}_{d_{r}=0}G_{\Delta}(\Ft)q^{k_{A}(\Ft)}
=G_{\Delta}(\Ft_{r-1})q^{k_{A}(\Ft_{r-1})+\mu_{r}-1}.
\end{equation}

If $\Fs$ is maximal, by this formula
$$
\sum_{\Ft\in\Omega(\Fs)}G_{\Delta}(\Ft)q^{k_{A}(\Ft)}=q^{\sum_{i=2}^{r} (\mu_{i}-1)}\sum_{\Ft'\in\Omega(\Fs_{1})}G_{\Delta}(\Ft')q^{k_{A}(\Ft')}.
$$
Referring to the case $r=1$, since $\Fs_{1}=\mu_{1}$ is maximal, 
$\sum_{\Ft\in\Omega(\Fs)}G_{\Delta}(\Ft)q^{k_{A}(\Ft)}=0.
$
If $\Fs$ is not maximal, by the definition of $\Omega(\Fs)$, we have
$$
\sum_{\Ft\in\Omega(\Fs)}G_{\Delta}(\Ft)q^{k_{A}(\Ft)}=
q^{\sum^{r}_{i=i_{\Box}+1}(\mu_{i}-1)}\sum_{\Ft'\in\Omega(\Fs_{i_{\Box}})} G_{\Delta}(\Ft')q^{k_{A}(\Ft')}.
$$
Referring to~\eqref{eq:Omega-pf-unbox}, we have
$$
\sum_{\Ft\in\Omega(\Fs)}G_{\Delta}(\Ft)q^{k_{A}(\Ft)}=\begin{cases}
q^{k_{A}(\Fs)-(r-i_{\Box})}(1-q^{-1})^{-1}G_{\Delta}(\Fs_{i_{\Box}}) & \text{ if $b_{i_{\Box}}>0$,}\\
q^{k_{A}(\Fs)-(r-i_{\Box})} G_{\Delta}(\Fs_{i_{\Box}}) &\text{ if $b_{i_{\Box}}=0$.}
\end{cases}
$$
This completes the proof of Lemma~\ref{pro:Omega-G-A}.
\end{proof}

To continue the analysis, it is convenient to separate the cases $k_r-\mu_r$ even and $k_r-\mu_r$ odd.  The even case turns out
to be much easier and we defer this.  So we suppose until further notice that $k_r-\mu_r$ is odd.
With this hypothesis $k_{r}\ne \mu_r+2\sum^{r-1}_{i=1}\mu_i$
and $\Omega^{=}_{B}(\mu,k_{r})=\emptyset$.  Let $\mu'=\upsilon(\mu)$
(cf.\ \eqref{upsilon}).
For $\Fs=(b_{1},\dots,b_{r})\in\Omega_{A}(\mu',k_{r})$,
let $\Delta_{C}(\Fs)=(\Fb_{1},\dots,\Fb_{2r-1})$ be the 
decorated array given as follows:  $\Fb_{1}=2k_{r}$ and,
for $1\leq i\leq r-1$, 
$\Fb_{i+1}=2k_{r}-\sum^{i}_{j=1}b_{j}$,
${\Fb}_{2r-i}=\sum^{i}_{j=1}b_{j}$.
The array is decorated following the same rules as above, with the 
$d_i$ replaced by $b_i$ and $\mu$ by $\mu'$.  (So for example 
when $1\leq i\leq r$ the entry $\Fb_{i}$ is boxed if $b_{i}=\mu_{i}'$.)
Note that this gives a
$\Delta_{C}$-decorated array in the sense of \cite{FZ}.

We associate different algebraic quantities to this array, ones associated with the values
of quadratic Gauss sums, and so ones that reflect the 
hidden role of the double cover in this combinatorial identity.
For the entry $a=\Fb_{i}$ ($1\leq i\leq 2r-1$), let
$$
\tilde{\gamma}(a)=\begin{cases}
1-q^{-1}& \text{ if $a$ is neither boxed nor circled and $a$ is even,}\\
-q^{-1}& \text{ if $a$ is boxed but not circled and $a$ is even,}\\
q^{-\frac{1}{2}} & \text{ if $a$ is boxed but not circled and $a$ is odd,}\\
1& \text{ if $a$ is circled but not boxed,}\\
0& \text{ if $a$ is both boxed and circled.}
\end{cases}
$$
Define $G_{\Delta_{C}}(\Fs)=\prod^{2r-1}_{i=1}\tilde{\gamma}(\Fb_{i})$.
Since $\Fb_{i+1}$ and ${\Fb}_{2r-i}$ have the same parity for $1\leq i\leq r-1$ and since $\Fb_{1}$ is even, it follows that $G_{\Delta_{C}}(\Fs)$ is in $\Z[q^{-1}]$.
Also note that $\tilde{\gamma}(b)=\gamma(a)$ if  $a$ and $b$ have the same decorations and $b$ is even.
This simple fact will be used repeatedly in the calculations below. 

\begin{lemma} \label{lm:C-total-parity}
Let $k$ a non-negative integer and suppose that $\Fs=(b_{1},\dots,b_{r})\in\Omega_{A}(\mu',k)$. 
\begin{enumerate}
\item\label{part1}
If $b_{r}<\mu_{r}$, then $G_{\Delta_{C}}(\Fs)=0$ unless $b_{i}$ is even for $1\leq i\leq r-1$.
\item\label{part2}
If $b_{r}=\mu_{r}$, then $G_{\Delta_{C}}(\Fs)=0$ unless $b_{i_{\Box}}\equiv k-\mu_{r}\bmod{2}$ and $b_{i}$ is even for $1\leq i\ne i_{\Box}\leq r-1$, 
where $i_{\Box}=i_{\Box}(\Fs)$.
\end{enumerate}
\end{lemma}

\begin{proof}
To prove part~\ref{part1}, suppose that $G_{\Delta_{C}}(\Fs)\ne 0$ and
$b_{r}<\mu_{r}$. 
Then the entry ${\Fb}_{r+1}$ in the $\Delta_{C}(\Fs)$-decorated array is not boxed and ${\Fb}_{r+1}$ is even. 
When $i=1$, $\Fb_1=2k_r$ is even.
For $2\leq i\leq r-1$, if $b_{i}<2\mu_{i}$, then ${\Fb}_{2r-i+1}$ is unboxed and even; 
if $b_{i}=2\mu_{i}$, then ${\Fb}_{2r-i+1}={\Fb}_{2r-i}-b_{i}$ is even. 
Hence ${\Fb}_{2r-i+1}$ is even if ${\Fb}_{2r-i}$ is even. 
Since ${\Fb}_{r+1}$ is even,  ${\Fb}_{2r-i}$ is even for each $1\leq i\leq r-1$. Since ${\Fb}_{2r-i}-{\Fb}_{2r-i+1}=b_{i}$, $b_{i}$ is even for  $1\leq i\leq r-1$.

As for part~\ref{part2}, suppose that $b_{r}=\mu_{r}$. 
 By the definition of $i_{\Box}$,  ${\Fb}_{2r-i-1}$ is boxed and $b_{i}=\mu'_{i}$ is even for $i_{\Box}<i\leq r-1$.
Since $b_{i_{\Box}}<2\mu_{i_{\Box}}$, 
similarly to the case $b_{r}<\mu_{r}$,
 $b_{i}$ and $\Fb_{i}$ are even for $1\leq i< i_{\Box}$. Since 
$$
k=\Fb_{i_{\Box}-1}+b_{i_{\Box}}+2\sum^{r-1}_{i=i_{\Box}+1}\mu_{i}+\mu_{r},
$$
it follows that $b_{i_{\Box}}\equiv k-\mu_{r}\bmod{2}$ and $\Fb_{i}\equiv k-\mu_{r}\bmod{2}$ for $i_{\Box}\leq i\leq r-1$.
\end{proof}
  
  \begin{lemma}\label{lm:disjoint}
Let $\Ft$ and $\Fs$ be distinct patterns in $\Omega^{\leq}_{A}(\mu,k)$.  Then $\Omega(\Ft)\cap \Omega(\Fs)=\emptyset$.
\end{lemma}

\begin{proof}
Let $\Ft=(d_{1},\dots,d_{r})$ and $\Fs=(b_{1},\dots,b_{r})$ be in $\Omega^{\leq}_{A}(\mu,k)$.
If $k=\sum^r_{i=1}\mu_i$, then $\Omega^{\leq}_{A}(\mu,k)$ contains only one element.
Without loss of generality, suppose that $k<\sum^r_{i=1}\mu_i$ and $i_{\Box}(\Ft)\geq i_{\Box}(\Fs)$.  

The first possibility is $i_{\Box}=i_{\Box}(\Ft)= i_{\Box}(\Fs)$.
If there exists $\Fx=(x_{1},\dots,x_{r})\in \Omega(\Ft)\cap \Omega(\Fs)$, then $x_{i}=b_{i}=d_{i}$ for $1\leq i< i_{\Box}$. Since $b_{i}=d_{i}=\mu_{i}$ for $i_{\Box}<i\leq r$ and $\sum^{r}_{i=1}b_{i}=\sum^{r}_{i=1}d_{i}=k$, we have
$b_{i_{\Box}}=d_{i_{\Box}}$.  Then $\Fs=\Ft$, which contradicts our hypothesis.

Suppose instead that $i_{\Box}(\Ft)\ne i_{\Box}(\Fs)$.
Let $(x_{1},\dots,x_{r})$ be in $\Omega(\Fs)$.  Then  $\sum^{i_{\Box}(\Fs)}_{j=1}x_{j}\leq k-\sum^{r}_{j=i_{\Box}(\Fs)+1}\mu_{j}$, and $x_{j}\leq \mu_{j}$ for $i_{\Box}(\Fs)<j\leq i_{\Box}(\Ft)-1$. Thus
$$
\sum^{i_{\Box}(\Ft)-1}_{j=1}x_{j}\leq k-\sum^{r}_{j=i_{\Box}(\Ft)}\mu_{j}.
$$
For $\Ft$, we have $\sum^{i_{\Box}(\Ft)-1}_{j=1}d_{j}>k-\sum^{r}_{j=i_{\Box}(\Ft)}\mu_{j}$ and 
then 
$$
\sum^{i_{\Box}(\Ft)-1}_{j=1}x_{j}\leq k-\sum^{r}_{j=i_{\Box}(\Ft)}\mu_{j}< \sum^{i_{\Box}(\Ft)-1}_{j=1}d_{j}.
$$
The inequality implies $(x_1,\dots,x_r)\notin \Omega(\Ft)$.
Indeed, if $(x_1,\dots,x_r)\in \Omega(\Ft)$, then $x_j=d_j$ for $1\leq j\leq i_{\Box}(\Ft)-1$.
Hence $\Omega(\Ft)\cap \Omega(\Fs)=\emptyset$.  
\end{proof}

Write $\mu^{T}=(\mu_{1},\dots,\mu_{r-1})$ and $k^{\circ}=\frac{k_{r}-(\mu_{r}+1)}{2}$.
\begin{lemma}\label{lemma:disjoint-union}
If $k_r<\mu_r+2\sum^{r-1}_{i=1}\mu_i$, then
we have the disjoint union
\begin{equation}\label{eq:odd-omega}
\Omega^{>}_{B}(\mu,k_{r})=\sqcup_{\Fs\in \Omega^{\leq}_{A}(\mu^{T},k^{\circ})}\Omega^{>}_{B}(\Fs,k_{r}),
\end{equation}
where
$$
\Omega^{>}_{B}(\Fs,k_{r})=\cpair{(\Fs',d_{r})\in\Omega^{>}(\mu)\mid \Fs'\in\Omega(\Fs),~ d_{r}=k_{r}-2k_{A}(\Fs')}.
$$
\end{lemma}
Notice that if $k_r>\mu_r+2\sum^{r-1}_{i=1}\mu_i$, then $\Omega_{B}(\mu,k_{r})=\Omega^{>}_{B}(\mu,k_{r})$.

\begin{proof}
Let $\Ft=(d_{1},\dots,d_{r})$ be  in $\Omega^{>}_{B}(\mu,k_{r})$. Then $k_{A}(\Ft_{r-1})=\sum^{r-1}_{i=1}d_i\leq k^{\circ}$. 
As $\sum^{r-1}_{i=1}d_i\leq k^{\circ}<\sum^{r-1}_{i=1}\mu_i$, in the decreasing sequence
$$
\mu_1+\cdots+\mu_{r-1},d_1+\mu_2+\cdots+\mu_{r-1},\cdots,d_1+\cdots+d_{r-2}+\mu_{r-1},d_1+\cdots+d_{r-2}+d_{r-1}
$$
there exists a minimal index $i_{0}$  such that  
\begin{equation}\label{eq:Omega-pf}
\sum^{i_{0}-1}_{j=1}d_{j}>k^{\circ}-\sum^{r-1}_{j=i_{0}}\mu_{j}\text{ and }
\sum^{i_{0}}_{j=1}d_{j}\leq k^{\circ}-\sum^{r-1}_{j=i_{0}+1}\mu_{j}.
\end{equation}
We have $i_{0}\geq 1$ as $k^{\circ}<\sum^{r-1}_{i=1}\mu_i$. 
Let $\Fs=(b_{1},\dots,b_{r-1})$ defined by $i_{\Box}(\Fs)=i_{0}$, $b_{i}=d_{i}$ for $1\leq i\leq i_{0}-1$ and $b_{i_{0}}=k^{\circ}-\sum^{r}_{j=i_{0}+1}\mu_{j}-\sum^{i_{0}-1}_{j=1}b_{j}$. 
By~\eqref{eq:Omega-pf}, $0\leq b_{i_{0}}<\mu_{i_{0}}$ and $d_{i_{0}}\leq b_{i_{0}}$. Thus $\Ft_{r-1}$ is in $\Omega(\Fs)$.
Then using Lemma \ref{lm:disjoint}, we obtain the disjoint union \eqref{eq:odd-omega}.\end{proof}

Define the set $\Omega^{=}_{A,\Box}(\mu',k_{r})$ to be
$$
\{\Fs=(b_{1},\dots,b_{r})\in\Omega^{=}_{A}(\mu',k_{r})\mid  b_{i_{\Box}(\Fs) }+k_{r}-\mu_{r} \equiv b_{i}\equiv 0 \bmod{2} \text{ for }1\leq i\ne i_{\Box}(\Fs) <r\},
$$
and
$$
\Omega^{<}_{A,\Box}(\mu',k_{r})=\cpair{\Fs=(b_{1},\dots,b_{r})\in\Omega^{<}_{A}(\mu',k_{r})\mid k_{r}-b_{r}\equiv b_{i}\equiv 0\bmod{2} \text{ for }1\leq i<r}.
$$

Let $\varrho\colon \Omega^{=}_{A,\Box}(\mu',k_{r})\to \Omega^{\leq}_{A}(\mu^{T},k^{\circ})$ be the map defined by
$$
\varrho\colon \Fs=(b_{1},b_{2},\dots,b_{r-1},\mu_{r})
\mapsto \left(\frac{b_{1}}{2},\dots,\frac{b_{i_{\Box}(\Fs )-1}}{2},\frac{b_{i_{\Box}(\Fs )}-1}{2},\frac{b_{i_{\Box}(\Fs )+1}}{2},\dots,\frac{b_{r-1}}{2}\right).
$$ 
Indeed, since $k_{r}-\mu_{r}$ is odd, $b_{i_{\Box}(\Fs)}-1$ is even.
As $k_{r}\not\equiv \mu_r\bmod{2}$, $k_r\ne 2\sum^{r-1}_{i=1}\mu_i+\mu_r$,
$\mu'$ is not in $\Omega^{=}_{A,\Box}(\mu',k_{r})$ and $\mu^{T}$ is not in $\Omega^{\leq}_{A}(\mu^{T},k^{\circ})$.
Thus $\varrho$ is a bijective map from $\Omega^{=}_{A,\Box}(\mu',k_{r})$ to $\Omega^{\leq}_{A}(\mu^{T},k^{\circ})$. 
For  $\Ft =(d_{1},\dots,d_{r-1})\in \Omega^{\leq}_{A}(\mu^{T},k^{\circ})$, $i_{\Box}(\Ft)>0$, and the inverse $\varrho^{-1}(\Ft)$ is 
given by
$$
(2d_{1},\dots,2d_{i_{\Box}(\Ft)-1},2d_{i_{\Box}(\Ft)}+1,2d_{i_{\Box}(\Ft)+1},\dots,2d_{r-1},\mu_{r}).
$$
This is in $\Omega^{=}_{A,\Box}(\mu',k_{r})$ as $2d_{i_{\Box}(\Ft )}+1<2\mu_{i_{\Box}(\Ft)}$.
Note that  $i_{\Box}(\Ft)=i_{\Box}(\varrho^{-1}(\Ft))$.

Extend the domain of $\varrho$ by defining $\varrho\colon \Omega^{<}_{A,\Box}(\mu',k_{r})\to \Omega^{<}_{B}(\mu,k_{r})$ by
$$
\varrho(b_{1},\dots,b_{r})\mapsto \ppair{\frac{b_{1}}{2},\dots,\frac{b_{r-1}}{2},b_{r}}.
$$
Again, $\varrho$ is a bijection.

By Lemma~\ref{lm:C-total-parity},
for $\Fs\in\Omega^{\leq}_{A}(\mu',k_{r})$,
 $G_{\Delta_{C}}(\Fs)=0$ unless $\Fs\in \Omega^{=}_{A,\Box}(\mu',k_{r})\cup \Omega^{<}_{A,\Box}(\mu',k_{r})$.

\begin{lemma}\label{lm:n-odd-dr}
If $k_{r}<\mu_{r}+2\sum^{r-1}_{i=1}\mu_{i}$ and $k_{r}-\mu_{r}$ is odd, 
then for $\Fs\in \Omega^{\leq}_{A}(\mu^{T},k^{\circ})$
\begin{equation}\label{eq:n-odd-dr}
G_{\Delta_{C}}(\varrho^{-1}(\Fs))=
\sum_{\Ft\in\Omega^{>}_{B}(\Fs,k_{r})}G(\Ft).
\end{equation}
Moreover,
$$
\sum_{\Ft\in\Omega_{B}(\mu,k_{r})}G(\Ft)
=\sum_{\Ft\in\Omega^{\leq}_{A}(\mu',k_{r})}G_{\Delta_{C}}(\Ft).
$$
\end{lemma}
\begin{proof}
If $k_{r}<\mu_{r}$, the proof is straightforward.
So we consider the case $k_{r}>\mu_{r}$. First, we have
$$
\sum_{\Ft\in\Omega_{B}(\mu,k_{r})}G(\Ft)=
\sum_{\Ft\in\Omega^{<}_{B}(\mu,k_{r})}G(\Ft)
+\sum_{\Ft\in\Omega^{>}_{B}(\mu,k_{r})}G(\Ft).
$$

Suppose that $\Ft=(d_{1},\dots,d_{r})$ is in $\Omega^{<}_{B}(\mu,k_{r})$. Let $\Fs=\varrho^{-1}(\Ft)$, which is in $\Omega^{<}_{A,\Box}(\mu',k_{r})$.
Then $G(\Ft)=G_{\Delta}(\Ft)$.
Let $\Delta(\Ft)=(\Fc_{1},\dots,\Fc_{2r-1})$ and $\Delta_{C}(\Fs)=(\Fb_{1},\dots,\Fb_{2r-1})$ be the corresponding decorated arrays. 
By definition,  for all $i$, $\Fb_{i}=2\Fc_{i}$, and $\Fb_{i}$ and $\Fc_{i}$ have the same decoration. 
Then $\gamma(\Fc_{i})=\tilde{\gamma}(\Fb_{i})$ for all $i$ and $G(\Ft)=G_{\Delta_{C}}(\Ft)$.
Since 
$$
\sum_{\Ft\in\Omega^{<}_{A}(\mu',k_{r})}G_{\Delta_{C}}(\Ft)=\sum_{\Ft\in\Omega^{<}_{A,\Box}(\mu',k_{r})}G_{\Delta_{C}}(\Ft)
=\sum_{\Fs\in\Omega^{<}_{B}(\mu,k_{r})}G(\Ft),
$$
we only need to verify Eqn.~\eqref{eq:n-odd-dr}.

By Lemma~\ref{lemma:disjoint-union},
$$
\sum_{\Ft\in\Omega^{>}_{B}(\mu,k_{r})}G(\Ft)=\sum_{\Fs\in\Omega^{\leq}_{A}(\mu^{T},k^{\circ})}\sum_{\Ft\in\Omega^{>}_{B}(\Fs,k_{r})}G(\Ft).
$$
For  $\Ft\in \Omega^{>}_{B}(\mu,k_{r})$, 
recall that
$$
G(\Ft)=(1-q^{-1})q^{-k^{\circ}-1}G_{\Delta}(\Ft_{r-1})q^{k_{A}(\Ft_{r-1})}.
$$
Thus for $\Fs\in\Omega^{\leq}_{A}(\mu^{T},k^{\circ})$, we have 
$$
\sum_{\Ft\in\Omega^{>}_{B}(\Fs,k_{r})}G(\Ft)
=(1-q^{-1})q^{-k^{\circ}-1}
\sum_{\Ft_{r-1}\in\Omega(\Fs)}G_{\Delta}(\Ft_{r-1})q^{k_{A}(\Ft_{r-1})}.
$$
By Lemma \ref{pro:Omega-G-A}, we have 
\begin{equation}\label{eq:n-odd-G-g}
\sum_{\Ft\in\Omega^{>}_{B}(\Fs,k_{r})}G(\Ft)=\begin{cases}
(1-q^{-1})q^{-(r-i_{\Box})}G_{\Delta}(\Fs_{i_{\Box}})& \text{ if }b_{i_{\Box}} =0\\
q^{-(r-i_{\Box})}G_{\Delta}(\Fs_{i_{\Box}})
& \text{ if }b_{i_{\Box}}>0,
\end{cases}
\end{equation}
where $i_{\Box}=i_{\Box}(\Fs)$.
Let $\Delta(\Fs_{i_{\Box}})=(\Fc_{1},\dots,\Fc_{i_{\Box}},\Fc_{2r+1-i_{\Box}},\dots,{\Fc}_{2r-1})$.
As the entry $\Fc_{i_{\Box}}$ is unboxed, 
$$
\gamma(\Fc_{i_{\Box}})=\begin{cases}
1& \text{ if } b_{i_{\Box}}=0\\
1-q^{-1}& \text{ if } 0<b_{i_{\Box}}<\mu_{i_{\Box}}.
\end{cases}
$$
Thus, using Eqn.~\eqref{eq:n-odd-G-g}, we obtain
\begin{equation}\label{eq:G-B}
\sum_{\Ft\in\Omega^{>}_{B}(\Fs,k_{r})}G(\Ft)=(1-q^{-1})q^{-(r-i_{\Box})}\prod^{i_{\Box}-1}_{i=1}\gamma(\Fc_{i})\gamma({\Fc}_{2r-i}).
\end{equation}

Let $\Fx=\varrho^{-1}(\Fs)$ and let $\Delta_{C}(\Fx)=(\Fb_{1},\dots,\Fb_{2r-1})$ be the $\Delta_{C}$-decorated array.
Recall that $\Fx=(2b_{1},\dots,2b_{i_{\Box}}+1,\dots,2b_{r-1},\mu_{r})$, where $b_{i}=\mu_{i}$ for $i_{\Box}<i\leq r$ and $0<2b_{i_{\Box}}+1<2\mu_{i_{\Box}}$. 
The entries $\Fb_{i+1}$ and ${\Fb}_{2r-i}$ are boxed and uncircled for $i_{\Box}\leq i\leq r-1$.
Since $\Fb_{i+1}$ and ${\Fb}_{2r-i}$ are odd for $i_{\Box}\leq i\leq r-1$, 
$$
\prod^{r-1}_{i=i_{\Box}} \tilde{\gamma}(\Fb_{i+1})\,\tilde{\gamma}({\Fb}_{2r-i})=q^{-(r-i_{\Box})}.
$$
As the entry $\Fb_{i_{\Box}}$ is unboxed, uncircled and even, $\tilde{\gamma}(\Fb_{i_{\Box}})=1-q^{-1}$.
Thus
\begin{equation}\label{eq:G-C}
G_{\Delta_{C}}(\Fx)=(1-q^{-1})q^{-(r-i_{\Box})}\prod^{i_{\Box}-1}_{i=1}\tilde{\gamma}(\Fb_{i})\,\tilde{\gamma}_{A}({\Fb}_{2r-i}).
\end{equation}
Since $\Fb_{i}$ and ${\Fb}_{2r-i}$ are even for $1\leq i\leq i_{\Box}-1$ and have the same decorations as $\Fc_{i}$ and ${\Fc}_{2r-i}$ respectively, we have
$\tilde{\gamma}(\Fb_{i})=\gamma(\Fc_{i})$ and $\tilde{\gamma}({\Fb}_{2r-i})=\gamma({\Fc}_{2r-i})$ for $1\leq i\leq i_{\Box}-1$.
Comparing with \eqref{eq:G-B} and \eqref{eq:G-C}, one obtains Eqn.~\eqref{eq:n-odd-dr}.
\end{proof}

Next, we obtain a formula for $\sum_{\Ft\in\Omega_{B}(\mu,k_{r})}G(\Ft)$ in terms of $\Delta_{C}$-decorated arrays. 
Here we consider both possible parities for $k_r-\mu_r$.
\begin{proposition}\label{pro:odd}
Let $k_{r}$ be a non-negative integer. 
Then
$$
\sum_{\Ft\in\Omega_{B}(\mu,k_{r})}G(\Ft)=\sum_{\Ft\in\Omega^{\leq}_{A}(\mu',k_{r})}G_{\Delta_{C}}(\Ft).
$$
\end{proposition}
\begin{proof}
We analyze separately three cases based on the comparison of the quantities $k_r$, $\mu_{r}+2\sum^{r-1}_{i=1}\mu_{i}$.

Suppose first that $k_{r}>\mu_{r}+2\sum^{r-1}_{i=1}\mu_{i}$. 
Then $\Omega_{B}(\mu,k_{r})=\Omega^{>}_{B}(\mu,k_{r})$ and the set $\Omega^{\leq}_{A}(\mu',k_{r})$ is empty.
We need only show that $\sum_{\Ft\in\Omega_{B}(\mu,k_{r})}G(\Ft)=0$. 
By Prop.~\ref{pro:G}, for $\Ft\in \Omega^{>}_{B}(\mu,k_{r})$, $G(\Ft)=0$ if $k_{r}-\mu_{r}$ is even. 
Suppose that $k_{r}-\mu_{r}$ is odd. 
For $\Ft=(\Fs,d_{r})\in\Omega^{>}_{B}(\mu,k_{r})$, let $\Delta(\Ft)=(\Fc_{1},\dots,\Fc_{2r-1})$ be the corresponding decorated array.
By Prop.~\ref{pro:G}, 
$$
\sum_{\Ft\in\Omega_{B}(\mu,k)}G(\Ft)=(1-q^{-1})q^{-\frac{(k_{r}-\mu_{r}+1)}{2}}\sum_{\Fs\in\Omega^{\leq}(\mu^{T})}G_{\Delta}(\Fs)q^{k_{A}(\Fs)}.
$$
By Lemma~\ref{pro:Omega-G-A},  $\sum_{\Fs\in\Omega^{\leq}(\mu^{T})}G_{\Delta}(\Fs)q^{k_{A}(\Fs)}=0$ and so 
$\sum_{\Ft\in\Omega_{B}(\mu,k)}G(\Ft)=0$, as desired.
 
Suppose that $k_{r}=\mu_{r}+2\sum^{r-1}_{i=1}\mu_{i}$.
Then
$$
\sum_{\Ft\in\Omega_{B}(\mu,k_{r})}G(\Ft)=G_{\Delta}(\mu)+\sum_{\Ft\in\Omega^{>}_{B}(\mu,k_{r})}G(\Ft)
=-q^{1-2r}+0=G_{\Delta_{C}}(\mu').
$$

Finally, suppose that $k_{r}<\mu_{r}+2\sum^{r-1}_{i=1}\mu_{i}$.
If $k_{r}-\mu_{r}$ is odd, the result follows from Lemma \ref{lm:n-odd-dr}. 
So supose that  $k_{r}-\mu_{r}$ is even. 
For $\Ft \in \Omega^{>}_{B}(\mu,k_{r})$, $G(\Ft)=0$.
We need only consider the set $\Omega^{\leq}_{B}(\mu,k_{r})$.

Let 
$$
\Omega^{\leq}_{A,e}(\mu',k_{r})=\{(b_{1},\dots,b_{r})\in \Omega^{\leq}_{A}(\mu',k_{r})\mid
b_{i}\text{ is even for }1\leq i<r\}.
$$
By Lemma~\ref{lm:C-total-parity}, $G_{\Delta_{C}}(\Fs)=0$ unless $\Fs\in\Omega^{\leq}_{A,e}(\mu',k_{r})$.
Let $\varrho'\colon \Omega^{\leq}_{A,e}(\mu',k_{r}) \to \Omega^{\leq}_{B}(\mu,k_{r})$ be the map
$$
\varrho'\colon (b_{1},\dots,b_{r})\mapsto
\ppair{\frac{b_{1}}{2},\dots,\frac{b_{r-1}}{2},b_{r}}.
$$
Then $\varrho'$ is bijective. 

Let $\Ft=(d_{1},\dots,d_{r})$ be in $\Omega^{\leq}_{B}(\mu,k_{r})$ and $\Fs=\varrho'^{-1}(\Ft)=(2d_{1},\dots,2d_{r-1},d_{r})$.
Let $\Delta(\Ft)=(\Fc_{1},\dots,\Fc_{2r-1})$ and $\Delta_{C}(\Fs)=(\Fb_{1},\dots,\Fb_{2r-1})$ be the corresponding decorated arrays. 
Then $\Fb_{i}=2\Fc_{i}$, and $\Fb_{i}$ and  $\Fc_{i}$ have the same decoration for all $i$. 
Thus we have $\gamma(\Fc_{i})=\tilde{\gamma}(\Fb_{i})$ for all $i$.
Hence $G_{\Delta_{C}}(\Fs)=G_{\Delta}(\Ft)$.
The Proposition follows. 
\end{proof}

\section{The General Case} \label{sec:general}

Our goal in this section is to reduce the general case to the totally resonant case.  A similar reduction was
carried out in the dual case in \cite{FZ}, Section 8 (and earlier for type A in \cite{BBF11}), and we both use and adapt those arguments here.  Recall that $\mu'=\upsilon(\mu)$.
Define 
 $CQ_{C}(\mu')$ (the set used to give the type C combinatorial quantities attached to $\mu'$) 
 to be the set of $(2r-1)$-tuples $(d_{1},\dots,d_{2r-1})$ of non-negative integers satisfying the inequalities
\begin{equation}
\begin{cases}
d_{j}\leq \mu'_{j} & 1\leq j\leq r,\\
d_{j+1}+d_{2r-j}\leq \mu'_{j+1}+d_{j}& 1\leq j\leq r-1.
\end{cases}
\end{equation}
(This is the set $CQ_{1}(\mu'')$ in \cite{FZ}, Eqn.~(25), with $\mu''_i=\mu'_{r+1-i}$.)
For $\Ft=(d_{1},\dots,d_{2r-1}) \in CQ_{C}(\mu')$, let $\Delta_{C}(\Ft)=(\Fc_{1},\dots,\bar{\Fc}_{1})$ (where $\bar{\Fc}_i:=\Fc_{2r-i}$ for convenience)
be the decorated array
with entries
\begin{equation*}
\bar{\Fc}_{j}=\sum^{j}_{i=1}d_{2r-i},~\Fc_{r}=\sum^{r}_{i=1}d_{2r-i}+d_{r}, \text{ and } \Fc_{j}=\Fc_{r}+\sum^{r-1}_{i=j}d_{i}, \text{ for } 1\leq j<r,
\end{equation*}
and decorations as follows:
\begin{enumerate}
\item The entry $\Fc_{j}$ is circled if $d_{j}=0$.
\item The entry $\Fc_{j}$ for $j\leq r$ is boxed if $d_{j}=\mu'_{j}$. The entry $\bar{\Fc}_{j}$ for $j<r$ is boxed if $d_{j+1}=\mu'_{j+1}+d_{j}-d_{2r-j}$.
\end{enumerate}

The weight vector ${\bf k}_{C}(\Ft)=(k_{1},\dots,k_{r})$, defined in \cite{FZ}, Eqn.~(26), is
\begin{equation*} 
k_{r}(\Ft)=\sum^{r}_{j=1} d_{2r-j} \text{ and }
k_{i}(\Ft)=\sum^{2r-1}_{j=i}d_{j}+d_{r}+\sum^{i-1}_{j=1}d_{2r-j}, \text{ for } 1\leq i<r.
\end{equation*}
(In this paper we use the phrase ``weighting vector" to avoid confusion with weights in the Lie-algebraic sense.)
Let $CQ_{C}(\mu',{\bf k}')$ be the subset of  $CQ_{C}(\mu')$ of short patterns $\Ft$ with ${\bf k}_C(\Ft)={\bf k}'$.
Define $G_{\Delta_{C}}(\Ft)=\prod^{2r}_{i=1}\tilde{\gamma}(\Fc_{i})$.
We remark that the arrays considered here are the same as those that arise in the type $C$ case treated in  \cite{FZ}, Section 6, but
the circling rule here is different than the one used in \cite{FZ}, Lemma 5.  However, it gives the same $G_{\Delta_{C}}(\Ft)$. 

Given a vector ${\bf k}=(k_1,\dots,k_r)\in \Z_{\geq0}^r$
we associate a graph on the index set $\{1,2,\dots,r\}$ as follows.
Each $i$ is a vertex.
Two vertices $i$ and $j$  are connected by an edge if and 
only if $j=i+1$ and $k_{i}=k_{j}$.  This is the graph whose edges correspond to those of  $\upsilon({\bf k})=(k_1',\dots,k_r')$ in \cite{FZ}, Section 8.
For a connected component $(i_1,i_1+1,\dots,i_2)$ of this graph, 
let $\CE=(k_{i_{1}},k_{i_{1}+1},\dots,k_{i_{2}})$  be the corresponding subsequence of ${\bf k}$,
also called a component of ${\bf k}$, and  $\CE'=(k'_{i_{1}},k'_{i_{1}+1},\dots,k'_{i_{2}})$ be the corrresponding
subsequence of $\upsilon({\bf k})$, also called a component of $\upsilon({\bf k})$.
Define $\ell_{\CE}=i_{1}$ and $r_{\CE}=i_{2}$.  Suppose the graph has $h$ connected components. 
They give a disjoint partition of $\upsilon({\bf k})$, 
$$
\upsilon({\bf k})=(2\CE_{1},\cdots,2\CE_{h-1},\upsilon(\CE_{h})),
$$
ordered by $\ell_{\CE_{i+1}}=r_{\CE_{i}}+1$ for $1\leq i<h$.

For each connected component $\CE$, $k_{i}$
satisfies the properties that 
$k_{i}=k_{j}$ for all $\ell_{\CE}\leq i,j\leq r_{\CE}$; either $k_{\ell_{\CE}-1}\ne k_{\ell_{\CE}}$ or $\ell_{\CE}=1$; and either $k_{r_{\CE}}\ne k_{r_{\CE}+1}$ or $r_{\CE}=r$. 
If $\ell_{\CE}=1$, let $a_{\CE}=0$, and if $r_{\CE}=r$, let $b_{\CE}=0$.  Otherwise let
$a_{\CE}=|k_{\ell_{\CE}-1}-k_{\ell_{\CE}}|$ and $b_{\CE}=|k_{r_{\CE}}-k_{r_{\CE}+1}|$.
If $r_{\CE}\neq \ell_{\CE}$, define $\mu(\CE)=(\mu(\CE)_{1},\mu(\CE)_{2},\dots,\mu(\CE)_{m(\CE)})$ by
specifying that $\mu(\CE)_{i}=\mu_{\ell_{\CE}+i}$ for all $1<i< m(\CE)$ and
that
$$\mu(\CE)_{1}=\begin{cases}\mu_{\ell_{\CE}}&\text{if 
$ k_{\ell_{\CE}-1}> k_{\ell_{\CE}}$}\\
\mu_{\ell_{\CE}}-a_{\CE}&\text{if  
$ k_{\ell_{\CE}-1}< k_{\ell_{\CE}}$,}
\end{cases}
$$
$$
\mu(\CE)_{m(\CE)}=\begin{cases}\mu_{r_{\CE}}-b_{\CE}&\text{if 
$k_{r_{\CE}}>k_{r_{\CE}+1}$}\\
\mu_{r_{\CE}}&\text{if $ k_{r_{\CE}}<k_{r_{\CE}+1}$.}
\end{cases}
$$

When $r_{\CE}=\ell_{\CE}\ne r$ (i.e.\ $m(\CE)=1$ and $\CE\ne \CE_{h}$), define
$$
\mu(\CE)_{1}=\begin{cases}
\mu_{r_{\CE}}-b_{\CE} &\text{if 
$ k_{r_{\CE}}> k_{r_{\CE}+1}$ and $ k_{\ell_{\CE}-1}> k_{\ell_{\CE}}$}\\
\mu_{r_{\CE}}-b_{\CE}-a_{\CE} &\text{if 
$ k_{r_{\CE}}> k_{r_{\CE}+1}$ and $ k_{\ell_{\CE}-1}< k_{\ell_{\CE}}$}\\
\mu_{r_{\CE}} &\text{if $ k_{r_{\CE}}< k_{r_{\CE}+1}$ and $ k_{\ell_{\CE}-1}> k_{\ell_{\CE}}$}\\
\mu_{r_{\CE}}-a_{\CE} &\text{if $ k_{r_{\CE}}< k_{r_{\CE}+1}$ and $ k_{\ell_{\CE}-1}< k_{\ell_{\CE}}$.}   
\end{cases}
$$
When $r_{\CE}=\ell_{\CE}=r$ (i.e.\ $m(\CE)=1$ and $\CE=\CE_{h}$), define
$$
(\mu_{\CE})_{1}=\begin{cases}
\mu_{r}-2a_{\CE}& \text{ if } k_{r-1}<k_{r}\\
\mu_{r}& \text{ if } k_{r-1}>k_{r}.
\end{cases}
$$

Following the treatment of type C in  \cite{FZ}, Section 8, let
$$
\Xi'_{\upsilon({\bf k})}=\cpair{{\bf x}'=(x'_{1},\dots,x'_{h})\in\Z^{h}_{\geq 0}\mid \sum^{h}_{i=1}x'_{i}+\sum^{h-1}_{i=1}(b_{\CE_{i}}-k_{r_{\CE_{i}}}+k_{r_{\CE_{i}}+1})=k_{r}}.
$$
Then there is a bijective map (see \cite{FZ}, Section 8)
$$
\Psi'_{\upsilon({\bf k})}\colon CQ_{C}(\mu',\upsilon({\bf k}))\mapsto
\bigsqcup_{{\bf x}'\in\Xi'_{\upsilon({\bf k})}}\prod^{h-1}_{i=1}\Omega^{\leq}_{A}(2\mu(\CE_{i}),x'_{i})\times\Omega_{A}(\upsilon(\mu(\CE_{h})),x'_{h}).
$$

\begin{lemma}\label{lm:odd-parity}
Let $\Ft=(d_{1},\dots,d_{2r-1})$ be in $CQ_{C}(\mu',\upsilon({\bf k}))$.
Then $G_{\Delta_{C}}(\Ft)=0$ unless $d_{i}$ and $d_{2r-i}$ are even for all $i$, $1\leq i< \ell_{\CE_{h}}$.
Moreover, $x'_{i}$ in the associated vector ${\bf x}'$ under the map $\Psi'_{\upsilon({\bf k})}$ is even for $i$, $1\leq i<h$.
\end{lemma}
\begin{proof}
Since $d_{i}\equiv d_{2r-i}\bmod{2}$ for all $i$, it is enough to show that $d_{2r-i}$ is even for all $i$, $1\leq i< \ell_{\CE_{h}}$.
If $h=1$, then $\ell_{\CE_{h}}=1$ and the statement holds. 
Suppose that $h>1$ and $G_{\Delta_{C}}(\Ft)\ne 0$.
By \cite{FZ}, Section 8, one of the entries $\bar{\Fc}_{\ell_{\CE_{h}}-1}$ and $\Fc_{\ell_{\CE}}$ is not boxed, and then 2 divides $\bar{\Fc}_{\ell_{\CE_{h}}-1}$ or $\Fc_{\ell_{\CE}}$.
Since $\bar{\Fc}_{\ell_{\CE_{h}}-1}+\Fc_{\ell_{\CE}}=2k_{\ell_{\CE_{h}}-1}$, $\bar{\Fc}_{\ell_{\CE_{h}}-1}$ and $\Fc_{\ell_{\CE}}$ are even. 
Assume that $\bar{\Fc}_{i}$ is  even for some $i\leq \ell_{\CE_{h}}-1$.
If $\bar{\Fc}_{i-1}$ is unboxed, then $2n$ divides $\bar{\Fc}_{i-1}$ with $n$, the degree of the cover in \cite{FZ}, equal to 1 here, and so $\bar{\Fc}_{i-1}$ is even. 
If $\bar{\Fc}_{i-1}$ is boxed, then $d_{2r-i}=2\mu_{i}+d_{i-1}-d_{2r-i+1}$ is even and $\bar{\Fc}_{i-1}=\bar{\Fc}_{i}-d_{2r-i+1}$ is even.
In sum, if $\bar{\Fc}_{i}$ is even, then $\bar{\Fc}_{i-1}$ is even. 
By induction, $\bar{\Fc}_{i}$ is even for all $i$, $1\leq i\leq \ell_{\CE_{h}}-1$.
By definition, $d_{2r-i}$ is even for all $i$, $1\leq i\leq \ell_{\CE_{h}}-1$ and $x'_{i}$ is even for all $i$, $1\leq i\leq h-1$.
\end{proof}

Next we show that we only need to consider vectors of the form $\upsilon({\bf k})$,
that is, those with the first $r-1$ entries even, in evaluating $G_{\Delta_{C}}(\Ft)$.
For an arbitrary vector ${\bf k}\in\Z^r_{\geq0}$, 
write $\ell({\bf k}):=\ell_{\CE_h}$, the index such that
$\ell_{\CE_h}=r$ when $k_{r-1}\ne 2k_r$, or
$k_{\ell_{\CE_h}-1}\ne k_{\ell_{\CE_h}}$ and $k_j=2k_r$ for $\ell_{\CE_h}\leq j<r$ when $k_{r-1}=2k_r$.

\begin{lemma}\label{lm:k-even}
Let ${\bf k}=(k_{1},\dots,k_{r})$ in $\Z^{r}_{\geq 0}$.
If $\Ft\in CQ_{C}(\mu',{\bf k})$, then $G_{\Delta_{C}}(\Ft)=0$ 
unless $k_{i}$ is even for all $i$, $1\leq i\leq r-1$.
\end{lemma}

\begin{proof}
If $h=1$, then $k_i=2k_{r}$ for all $i$, $1\leq i<r$. So we need only consider the case $h>1$ and $G_{\Delta_C}(\Ft)\ne 0$.
For $\Ft\in CQ_{C}(\mu',{\bf k})$, let $\Delta_{C}(\Ft)=(\Fc_{1},\dots,\bar{\Fc}_{1})$.
First, we show that $\Fc_i$ is even for all $i$, $1\leq i\leq \ell({\bf k})$.
 When $i=\ell({\bf k})$, we have $\Fc_i+\bar{\Fc}_{i-1}=2k_r$.
Since one of the entries $\Fc_i$ and $\bar{\Fc}_{i-1}$ is not boxed, either $\Fc_i$ or $\bar{\Fc}_{i-1}$ is even.
Since $\Fc_i+\bar{\Fc}_{i-1}=2k_r$, we conclude that in fact both $\Fc_i$ and $\bar{\Fc}_{i-1}$ are even. 
Suppose that $i<\ell({\bf k})$. Then $i<r$.
If $\Fc_i$ is not boxed, then 2 divides $\Fc_i$ by \cite{FZ}, Lemma 9.
If $\Fc_i$ is boxed, then necessarily $\Fc_i-\Fc_{i+1}=d_i=2\mu_i$, and hence $\Fc_i\equiv\Fc_{i+1}\bmod 2$.
Hence $\Fc_i$ is even for all $i$, $1\leq i\leq\ell({\bf k})$
and $d_i$ is even for all $i$, $1\leq i<\ell({\bf k})$.

Now, since $\Fc_1=k_1$ is even,
it is enough to show that for $1<i<r$ either $k_i$ or $k_{i-1}-k_i$ is even. 
If $\bar{\Fc}_{i-1}$ is not boxed, then $\bar{\Fc}_{i-1}$ is even.
Since $2\mid \Fc_{i}$, we see that $k_i=\Fc_i+\bar{\Fc}_{i-1}$ is even.
If $\bar{\Fc}_{i-1}$ is boxed, then $d_i=2\mu_i+k_{i-1}-k_i$.
Since $d_i$ is even for $1\leq i<\ell({\bf k})$, $k_{i-1}-k_i$ is even. 
When $\ell({\bf k})\leq i<r$, $k_i=2k_r$ is even. 
\end{proof}

Define $\Xi_{\bf k}$ to be the subset of $\Z^{h}_{\geq 0}$ consisting of the vectors $(x_{1},x_{2},\cdots,x_{h})$ satisfying
$$
x_{i}\leq \sum^{m(\CE_{i})}_{j=1}\mu(\CE_i)_{j} \text{ for $1\leq i<h$, and }
\sum^{h-1}_{i=1}(2x_{i}+b_{\CE_{i}})+x_{h}=k_{1}.
$$
Similarly to \cite{FZ},  Lemma 18, we have a bijective map
$$
\Psi_{{\bf k}}\colon CQ_{1}(\mu,{\bf k})\to \bigsqcup_{{\bf x}\in\Xi_{\bf k}}\prod^{h-1}_{i=1}\Omega^{\leq}_{A}(\mu(\CE_{i}),x_{i})\times\Omega_{B}(\mu(\CE_{h}),x_{h})
$$
given by $\Psi_{\bf k}((d_{1},\dots,d_{2r-1}))=(\Ft(\CE_{1}),\dots,\Ft(\CE_{r}))$, where
$$
\Ft(\CE_{i})=(d_{\ell_{\CE_{i}}},d_{\ell_{\CE_{i}}+1},\dots,d_{r_{\CE_{i}}-1}, d'_{r_{\CE_{i}}}).
$$
Here $d'_i=\min\{d_i,d_{2r-i}\}$ as in \cite{FZ}.
(In particular, if $\ell_{\CE}=r_{\CE}=1$, then $\Ft(\CE)=(d'_{1})$.)
It is easy to see that $\Psi_{{\bf k}}$ restricts to a bijection from $CQ^{\leq}_{1}(\mu,{\bf k})$ to
$$
\bigsqcup_{{\bf x}\in\Xi_{\bf k}}\prod^{h-1}_{i=1}\Omega^{\leq}_{A}(\mu(\CE_{i}),x_{i})\times\Omega^{\leq}_{B}(\mu(\CE_{h}),x_{h}).
$$

Also, let $v$ be the injective map $\upsilon$ on $\Xi_{\bf k}$  given by
$$
\upsilon\colon (x_{1},\dots,x_{h-1},x_{h}) \mapsto
(2x_{1},\dots,2x_{h-1},x_{h}).
$$
It is easy to verify that $\upsilon({\bf x})$ is in $\Xi'_{\upsilon({\bf k})}$.
By Lemma~\ref{lm:odd-parity}, $G_{\Delta_{C}}(\Ft)$ is supported on the patterns $\Ft$ whose associated vector ${\bf x}'$
is in $\upsilon(\Xi_{{\bf k}})$.

For $\CE=\CE_i$ with $r_\CE<r$, define
$$
c_\CE=\begin{cases}
\bar{c}_{r_\CE} & \text{ if } k_{r_\CE}>k_{r_\CE+1},\\
c_{r_\CE} & \text{ if } k_{r_\CE}<k_{r_\CE+1}.
\end{cases}
$$
No entry $c_{\CE_{i}}$   is either boxed  or circled 
except  when $\mu_{r_{\CE}}=k_{r_{\CE}}-k_{r_{\CE}+1}$.
For ${\bf x}=(x_{1},\dots,x_{h})$ in $\Xi_{\bf k}$, by \cite{FZ}, Section 8,  we have
$$
\sum_{\Ft\in\prod^{h-1}_{i=1}\Omega^{\leq}_{A}(2\mu(\CE_{i}),2x_{i})\times \Omega^{\leq}_{A}(\upsilon(\mu(\CE_{h})),x_{h})}
G_{\Delta_C}(\Ft)=
\prod^{h-1}_{i=1}
\sum_{\Ft(\CE_{i})}G_{\Delta_C}(\Ft(\CE_{i}))
\cdot\tilde{\gamma}(c_{\CE_{i}}) 
\times \sum_{\Ft(\CE_{h})} G_{\Delta_{C}}(\Ft(\CE_{h})),
$$
where $\Ft(\CE_{i})$ runs over $\Omega^{\leq}_{A}(2\mu(\CE_{i}),2x_{i})$ when $i\ne h$ and $\Omega_{A}(\upsilon(\mu(\CE_{h})),x_{h})$ when $i=h$.  Similarly,
$$
\sum_{\Ft\in\prod^{h-1}_{i=1}\Omega^{\leq}_{A}(\mu(\CE_{i}),x_{i})\times\Omega_{B}(\mu(\CE_{h}),x_{h})}
G(\Ft)=
\prod^{h-1}_{i=1}
\sum_{\Ft(\CE_{i})}G_{\Delta}(\Ft(\CE_{i}))
\cdot \gamma(\Fc_{\CE_{i+1}})
\times\sum_{\Ft(\CE_{h})}G(\Ft(\CE_{h})),
$$
where now $\Ft(\CE_{i})$ runs over $\Omega^{\leq}_{A}(\mu(\CE_{i}),x_{i})$ when $i\ne h$ and $\Omega_{B}(\mu(\CE_{h}),x_{h})$ when $i=h$.

By Lemma~\ref{lm:odd-parity}, 
for $i\ne h$,
$G_{\Delta_C}(\Ft'(\CE_{i}))$ is supported on the patterns $\Ft'(\CE_{i})$
which are in the image of the map  
 $\Ft(\CE_{i})\mapsto 2\Ft(\CE_{i})$ from $\Omega^{\leq}_{A}(\mu(\CE_{i}),x_{i})$ 
to $\Omega^{\leq}_{A}(2\mu(\CE_{i}),2x_{i})$.
Also by definition, $G_{\Delta}(\Ft(\CE_{i}))=G_{\Delta_C}(2\Ft(\CE_{i}))$.
Suppose $\CE=\CE_{h}$.
Let $\Ft(\CE)=(d_{1},\dots,d_{m(\CE)})$.
If $x_{h}-\mu_{r}$ is even then, by Lemma~\ref{lm:C-total-parity}, $G_{\Delta_{C}}(\Ft(\CE))=0$ unless $d_{i}$ is even for $1\leq i<m(\CE)$.
Define the map $\upsilon$ from $\Omega^{\leq}_{B}(\mu(\CE),x_{h})$ to $\Omega^{\leq}_{A}(\mu(\CE'),x_{h})$ by
$
\upsilon\colon (b_{1},\dots,b_{m(\CE)})\to 
(2b_{1},\dots,2b_{m(\CE)-1},b_{m(\CE)}).
$
Then 
$G(\Ft(\CE))=G_{\Delta_{C}}(\upsilon(\Ft(\CE)))$ by Prop.~\ref{pro:G}.
If $x_{h}-\mu_{r}$ is odd, then $x_{h}$ is not maximal
and, by Prop.~\ref{pro:odd}, 
$$
\sum_{\Ft(\CE)\in \Omega_{B}(\mu(\CE),x_{h})}G (\Ft(\CE))=\sum_{\Ft(\CE)\in \Omega^{\leq}_{A}(\upsilon(\mu(\CE)),x_{h})}G_{\Delta_{C}}(\Ft(\CE)).
$$
Then, for all $i$
\begin{equation*} 
\sum_{\Ft(\CE_{i})}G_{\Delta_C}(\Ft(\CE_{i}))=\sum_{\Ft(\CE_{i})}G(\Ft(\CE_{i})),
\end{equation*}
where on  the left-hand side, $\Ft(\CE_{i})$ runs over $\Omega^{\leq}_{A}(2\mu(\CE_{i}),2x_{i})$ when $i\ne h$ and over $\Omega_{A}(\upsilon(\mu(\CE_{h})),x_{h})$ when $i=h$;
on the right-hand side, $\Ft(\CE_{i})$ runs over $\Omega^{\leq}_{A}(\mu(\CE_{i}),x_{i})$ when $i\ne h$ and over $\Omega_{B}(\mu(\CE_{h}),x_{h})$ when $i=h$.

Since
$$
\sum_{\substack{\Ft\in CQ_{1}(\mu)\\ k_{B}(\Ft)={\bf k}}}G (\Ft)=\sum_{{\bf x}\in\Xi_{\bf k}}\sum_{\Ft}G(\Ft)
\quad\text{and}\quad
\sum_{\substack{\Ft\in CQ_{C}(\upsilon(\mu))\\ k_{C}(\Ft)=\upsilon({\bf k})}}G_{\Delta_C}(\Ft)=
\sum_{\upsilon({\bf x})\in\upsilon(\Xi_{{\bf k}})}\sum_{\Ft}G_{\Delta_C}(\Ft),
$$
we arrive at the following identity.
\begin{proposition} \label{pro:general}
$$
\sum_{\substack{\Ft\in CQ_{1}(\mu)\\ k_{B}(\Ft)={\bf k}}}G(\Ft)
=\sum_{\substack{\Ft\in CQ_{C}(\upsilon(\mu))\\ k_{C}(\Ft)=\upsilon({\bf k})}}G_{\Delta_C}(\Ft).
$$
\end{proposition}

\section{Proof of the Main Theorem}\label{characters}

Using these results, we now prove Theorem~\ref{main-theorem}.
We first combine Proposition~\ref{pro:general} and Theorem~\ref{thm1} to  obtain an inductive formula for
the prime power coefficients $H(p^{\bf k};p^{\lam})$ that arise in the Whittaker coefficient $W_{\bf m}(f,{\bf s})$.
This gives the following equality.
\begin{lemma}\label{lm:inductive-H}
Let ${\bf k}$, $\lam$ be in $\Z^r_{\geq 0}$.  Then
$$
H(p^{\bf k};p^{\lambda})=\sum_{\substack{{\bf k}',{\bf k}''\\ {\bf k}'+(0,{\bf k}'')=\upsilon({\bf k})}}
\sum_{\substack{\Ft\in CQ_{C}(\upsilon(\lam+\rho))\\ k_{C}(\Ft)={\bf k}'}}q^{k'_{r}+\frac{1}{2}\sum^{r-1}_{i=1}k'_{i}}G_{\Delta_C}(\Ft)H(p^{\upsilon^{-1}({\bf k}'')},p^{\nu})
$$
where the outer sum is over vectors ${\bf k}'=(k'_{1},\dots,k'_{r})$, 
 and ${\bf k}''=(k''_{1},\dots,k''_{r-1})$ of non-negative integers
 with $k'_{i}$  even for $1\leq i\leq r-1$ such that ${\bf k}'+(0,{\bf k}'')=\upsilon({\bf k})$  and 
$$
\nu=(\lam_{2}+\frac{k'_{1}}{2}+\frac{k'_{3}}{2}-k'_{2},\dots,\lam_{r-1}+\frac{k'_{r-2}}{2}+k'_{r}-k'_{r-1},\lam_{r}+k'_{r-1}-2k'_{r}).
$$
\end{lemma}
Note that since $k'_{i}$ is even for $1\leq i\leq r-1$ and ${\bf k}'+(0,{\bf k}'')=\upsilon({\bf k})$, we also have $k''_i$ is even for $1\le i\leq r-2$.  Thus $\upsilon^{-1}({\bf k}'')$ is a
vector of nonnegative integers, and the right-hand-side is well-defined.  (In fact, 
by Lemma \ref{lm:k-even} the factor $G_{\Delta_{C}}(\Ft)$ is also zero unless $k'_{i}$ is even for $1\leq i\leq r-1$.)

Recall that 
$GT^{\circ}(\mu')$  denotes the subset of GT-patterns $P$ in $GT(\mu')$ satisfying $\Fc(x)\equiv 0\bmod{2}$
for all generic entries $x$ in $P$.
A straightforward calculation shows that
\begin{lemma}\label{lm:GT}
A GT-pattern $P$ is in $GT^{\circ}(\mu')$ if and only if $P$ satisfies the following two conditions:
\begin{enumerate}
\item $a_{i,j}\equiv \mu_r \bmod{2}$;
\item In each row of $b$'s, there is at most one entry $b_{i,j_0}\not\equiv \mu_r\bmod{2}$. 
If such an entry exists, then $b_{i,j}=a_{i,j}=a_{i-1,j}$ for all $j$ such that $j_0<j\leq r$.
\end{enumerate}
\end{lemma}

We now turn to the proof of Theorem~\ref{main-theorem}. In Proposition~\ref{first-formula}, we have 
seen that 
$$
D_B(z;-q^{-1})\chi_{\lam}(z)
=\sum_{\bf k}H^{\flat}(p^{\bf k};p^{\lam})z_1^{k_1-\frac{a_{0,1}}{2}}z_2^{k_2-k_1-\frac{a_{0,2}}{2}}
\cdots z_r^{k_{r}-k_{r-1}-\frac{a_{0,r}}{2}}
$$
where $(a_{0,1},\dots,a_{0,r})$ is the top row of the GT-patterns in $GT(\upsilon(\lam+\rho))$ and
 $H^{\flat}(p^{\bf k};p^{\lam})=q^{-\sum_{i=1}^r k_i} H(p^{\bf k};p^{\lam})$.
On the right-hand side of Eqn.~\eqref{eq:T-B}, we have
\begin{equation}
\sum_{P\in GT^{\circ}(\upsilon(\lam+\rho))}G(P)z^{-\frac{\wt(P)}{2}}=
\sum_{\wt}\sum_{\substack{P\in GT^{\circ}(\upsilon(\lam+\rho))\\wt(P)=-wt}}G(P)z^{\frac{\wt(P)}{2}}.
\end{equation}
Note that by Lemma 10 the outer sum only runs through over vectors $\wt$ such that $\wt_i-a_{0,i}$ is even.
Now let $\wt$ be the specific vector (depending on ${\bf k}$, $\mu$) with coordinates
$\wt_i=2k_{i}-2k_{i-1}-a_{0,i}$ for $1\leq i\leq r$ (with $k_0=0$).
Set $t=-q^{-1}$. 
To complete the proof of Theorem~\ref{main-theorem}, it is sufficient to show that
\begin{equation}\label{eq:G=H}
H^{\flat}(p^{\bf k};p^{\bf m})=\sum_{\substack{P\in GT^{\circ}(\upsilon(\lam+\rho))\\ \wt(P)=-\wt}}G(P).
\end{equation}
Indeed, once we have established this, we will have shown that for fixed $q$ both sides of \eqref{eq:T-B} agree as polynomials
in the $z_i$.  Then each coefficient of their difference is a polynomial in $t$ with infinitely many roots, hence is zero.

We have already exhibited an inductive formula for the left-hand side of \eqref{eq:G=H}.
To establish this equality, we give a matching inductive formula for the right-hand side. 
This is based on the type C theory in \cite{FZ}.
Recall that a short GT-pattern of type C is an array of non-negative integers
$$
P_1=\begin{pmatrix}
a_{0,1}&&a_{0,2}&&\cdots&&a_{0,r}&\\
&b_{1,1}&&b_{1,2}&\cdots&b_{1,r-1}&&b_{1,r}\\
&&a_{1,2}&&\cdots&&a_{1,r}&
\end{pmatrix}
$$
such that the rows interleave.
Denote $GT_1(\mu')$ be the set of short GT-patterns of type C with the top row \eqref{eq:top-row}.
Similarly, we have the subset $GT^{\circ}_1(\mu')$. 
For ${\bf k}'\in\Z^r_{\geq0}$, let $GT^{\circ}_1(\mu',{\bf k}')$ be the set of short GT-patterns of type C with bottom row
\begin{equation}
(a_{0,2}+k'_1-k'_2,\dots,a_{0,r-1}+k'_{r-2}-k'_{r-1},a_{0,r}+k'_{r-1}-2k'_r),
\end{equation}
and such that $\sum^r_{i=1}a_{0,i}-2\sum^r_{i=1}b_{1,i}+\sum^r_{i=2}a_{1,i}=a_{0,1}-k'_1$.
For each short GT-pattern $P_1$, we similarly associate a polynomial $G(P_1)$.
Then assigning to each $P$ a pair $(P_1,Q)$ where $P_1$ consists of the top three rows of $P$ and $Q$ is
the GT pattern of rank one less obtained by removing the top two rows, we obtain an inductive formula
$$
\sum_{\substack{P\in GT^{\circ}(\mu')\\ \wt(P)=-\wt}}G(P)
=\sum_{\substack{{\bf k}',{\bf k}''\\ {\bf k}'+(0,{\bf k}'')= \upsilon({\bf k})}}
\sum_{P_1\in GT^{\circ}_{1}(\mu',{\bf k}')}G(P_1)\sum_{\substack{Q\in GT^{\circ}(\mu'')\\ \wt(Q)=-(\wt_2,\dots,\wt_r)}}G(Q).
$$
where the outer sum on the right-hand side is over vectors ${\bf k}'=(k'_{1},\dots,k'_{r})$  and ${\bf k}''=(k''_{1},\dots,k''_{r-1})$ of non-negative integers with $k'_{i}$ even for $1\leq i\leq r-1$
such that ${\bf k}'+(0,{\bf k}'')=\upsilon({\bf k})$,  and where
$$
\mu''=(2\mu_{2}+k'_{1}+k'_{3}-2k'_{2},\dots,2\mu_{r-1}+k'_{r-2}+2k'_{r}-2k'_{r-1},\mu_{r}+k'_{r-1}-2k'_{r}).
$$

Comparing to the inductive formula in Lemma \ref{lm:inductive-H}, to verify \eqref{eq:G=H}, it is equivalent to show that \eqref{eq:G=H} holds for $r=1$ and also that
\begin{equation}\label{last-step}
\sum_{P_1\in GT_{1}(\mu',{\bf k}')}G(P_1)=
\sum_{\substack{\Ft\in CQ_C(\mu')\\ k_C(\Ft)={\bf k}'}}G_{\Delta_C}(\Ft).
\end{equation}
When $r=1$, it is easy to verify the identity directly.  So it suffices to establish \eqref{last-step}.

For each $P_1\in GT_1(\mu',{\bf k}')$, let $d_r=b_{1,r}$ and, for $1\leq j\leq r-1$, let $d_j=b_{1,j}-a_{0,j+1}$  and $d_{2r-j}=b_{1,j}-a_{1,j+1}$.
Set $\Ft=(d_1,\dots,d_{2r-1}$).
Then the map $P_1\to \Ft$ establishes a bijection from $GT_1(\mu',{\bf k}')$ to $CQ_C(\mu',{\bf k}')$.
Let $\Delta_C(\Ft)=(\Fc_1,\dots,\Fc_{2r-1})$.
We have $\Fc(b_{1,j})\equiv\Fc_j\bmod{2}$ for $1\leq j\leq r$ and $\Fc(a_{1,j})=\Fc_{2r-j}$ for $1\leq j<r$. 
Hence an entry $x$ in $P_1$ is maximal or minimal if and only if the corresponding entry $\Fc_j$ in $\Delta_C(\Ft)$ is boxed or circled. 
Therefore, for such $P_1$ and $\Ft$, $G(P_1)=G_{\Delta_C}(\Ft)$.   Hence the equality \eqref{last-step} is true via term-by-term matching.
The Theorem follows. \qed

We close this Section with an example: $\lam=(3,2)$.  By Fulton and Harris~\cite{FultonHarris}, pg.\ 408, 
one has
$$
\chi_{\lam}(z)=\frac{
	\det\begin{pmatrix}
		z^{\frac{11}{2}}_1-z^{-\frac{11}{2}}_1& z^{\frac{11}{2}}_2-z^{-\frac{11}{2}}_2\\ 
		z^{\frac{3}{2}}_1-z^{-\frac{3}{2}}_1& z^{\frac{3}{2}}_2-z^{-\frac{3}{2}}_2
	\end{pmatrix}
}{
	\det\begin{pmatrix}
		z^{\frac{3}{2}}_1-z^{-\frac{3}{2}}_1& z^{\frac{3}{2}}_2-z^{-\frac{3}{2}}_2\\ 
		z^{\frac{1}{2}}_1-z^{-\frac{1}{2}}_1& z^{\frac{1}{2}}_2-z^{-\frac{1}{2}}_2
	\end{pmatrix}
} ,
$$
while the Weyl denominator is given by
$$
D_{B}(z;t)=z_{1}^{-\frac{3}{2}}z_{2}^{-\frac{1}{2}}(1+tz_{1})(1+tz_{1}z_{2}^{-1})(1+tz_{1}z_{2})(1+tz_{2}).
$$
The coefficient of $z_2^{11/2}$ in $D_{B}(z;t)\chi_\lam(z_1,z_2)$ is
$$
z_1^{3/2}t^3+z_1^{1/2}\left(t^3+t^2\right)+z_1^{-1/2}(t^3+t^2)+z_1^{-3/2}t^2,
$$
and the corresponding $GT$-patterns ($\wt_2=-11$, $\wt_1=3-2b_{1,2}$ with $b_{1,2}=3,2,1,0$) are
$$
\begin{pmatrix}
	11&&3&\\ &11&&b_{1,2}\\ &&11&\\ &&&11
\end{pmatrix}.
$$

\section{A Formula in Terms of Symplectic Shifted Tableaux}

In this last Section we recast the result above in terms of symplectic shifted tableaux.
Let us recall the standard symplectic tableaux of shape $\mu'$ associated to a GT-pattern in $GT(\mu')$
(see Hamel and King \cite{HK}, Definition 2.5).
Consider $\mu'$ as a partition with parts $\mu'_i$.
A shifted Young diagram defined by $\mu'$ consists of $|\mu'|=\sum^r_{i=1}\mu'_i$ boxes arranged in $r$ rows of lengths $\mu'_i$ with the rows left justified to a diagonal line.
Let $A=\{1,2,\dots,r\}\cup \{\bar{1},\bar{2},\dots,\bar{r}\}$ with ordering $\bar{1}<1<\bar{2}<\cdots<\bar{r}<r$. 
Define $ST^{\mu'}(sp(2r))$ to be the set of shifted Young diagrams of shape $\mu'$ which are filled
by placing an entry from $A$ in each of the boxes so that the entries are weakly increasing from left to right across each row and from top to bottom across each column, 
and are strictly increasing from top-left to bottom-right along each diagonal. 

Referring to \cite{HK}, Definition~5.2 or Beineke, Brubaker and Frechette \cite{BeBrFr}, Section~5.2, there is a bijection between 
the set of symplectic shifted tableaux of shape $\mu'$ and the set of strict GT-patterns of highest weight $\mu'$. 
Given a GT-pattern $P$,  the tableaux $S_P$ is obtained by the following rules:
the entry $a_{i-1,j}$ (resp.~$b_{i,j}$) of $P$ counts the number of boxes in the $(j-i+1)$-th row of $S_P$
whose entries are less than equal to the value $r-i+1$ (resp.~$\overline{r-i+1}$).

Let $ST^{\mu'}_\circ$ be the subset of $ST^{\mu'}(sp(2r))$ 
corresponding to $GT^\circ(\mu')$ under this bijection.
Then a tableau in $ST^{\mu'}(sp(2r))$ is in $ST^{\mu'}_\circ$ if and only if it satisfies the following two conditions.
\begin{enumerate}
	\item For each $m=2,\dots,r$, the sum of the numbers of $m$ and $\bar{m}$ in every row except the $m$-th row is even.
	\item For each $m=1,\dots,r$, there is at most one row containing an odd number of the entries $m$. 
	If there is one, when this row is above the $m$-th row, there is no $m$ occurring below this row and the entries $\bar{m}$ in and below this row are connected; when this row is the $m$-th row, there is no further restriction. 
\end{enumerate}
For each $S\in ST^{\mu'}_\circ$ and $m=1,\dots,r$, let $l_S(m)$ be the number of the row which contains an odd number of the entries $m$.
If there is no row containing an odd number of the entries $m$, we let $l_S(m)=m$.
Then it is easy to see that $l_S(1)=1$ and $l_S(m)\leq m$ for all $S$. 

Given $S\in ST^{\mu'}_\circ$, we define the following additional statistics.  Let 
$str(S)$ be the total number of connected components of all the ribbon strips of $S$ (see \cite{HK}, Definition 3.2);
$x_m(S)$ (resp.\ $x_{\bar{m}}(S)$) be  the number of entries $m$ (resp.\ $\bar{m}$) in $S$;
$$
\wt(S)=(x_r(S)-x_{\bar{r}}(S),x_{r-1}(S)-x_{\overline{r-1}}(S),\dots,x_1(S)-x_{\bar{1}}(S));
$$
$row_m(S)$ (resp.\ $row_{\bar{m}}(S))$ be the number of rows of $S$ containing the entry $m$ (resp.\ $\bar{m}$);
$con_{\bar{m}}(S)$ be the number of connected components of the ribbon strips of $S$ consisting solely of entries $\bar{m}$; and
$$
\overline{hgt}(S)=\sum^r_{m=1}(row_{\bar{m}}(S)-con_{\bar{m}}(S)-row_{m}(S)).
$$

\begin{example}
Let $P$ be the GT-pattern \eqref{pattern}.
Then the corresponding symplectic shifted tableau $S_P$ is
$$
\begin{ytableau}
1& 1 & 1 & \bar{2} & \bar{2} &3&3&4&4& \bar{5} & \bar{5} \\
\none&\bar{2}&\bar{2}&2&3&3&\bar{5}&\bar{5}&5&5&\none\\
\none&\none&\bar{3}&\bar{4}&\bar{4}&\bar{5}&5&5&5&\none&\none\\
\none&\none&\none&4&\bar{5}&\bar{5}&\none&\none&\none&\none&\none\\
\none&\none&\none&\none&\bar{5}&\none&\none&\none&\none&\none&\none
\end{ytableau}
$$
For this tableau, $str(S)=13$, $l_S(5)=3$ and $l_S(m)=m$ for $1\leq m\leq 4$. 
\end{example}

Using these statistics, we may express the character value $\chi_\lambda(z)$ in terms of symplectic shifted tableaux as follows.

\begin{corollary}\label{tableaux-result}
Let $\lam$ be a dominant weight of the Lie algebra $\mathfrak{so}_{2r+1}$ and $\rho$ be the sum of the fundamental weights. Then
$$
D_B(z;t)\chi_{\lambda}(z)=\sum_{S\in ST^{\upsilon(\lam+\rho)}_{\circ}}(-1)^{\frac{r(r+1)}{2}-l_S}
t^{\overline{hgt}(S)+l_S}(1+t)^{str(S)-r}z^{-\frac{\wt(S)}{2}}.
$$	
\end{corollary}

\begin{proof}
Let $\mu'=\upsilon(\lam+\rho)$, 
$P$ be a GT-pattern in $GT^{\circ}(\mu')$ and $S_P$ be the corresponding tableau in $ST^{\mu'}_\circ$.
Following an argument similar to \cite{BeBrFr}, Lemma 20, we have
$$
gen(P)=str(S_P)-r,~max(P)=\overline{hgt}(S_P)+\frac{r(r+1)}{2}, \text{ and } 
\frac{max_1(P)}{2}=\frac{r(r+1)}{2}-l_{S_P}.
$$
By \eqref{eq:T-B}, we obtain the Corollary.

\end{proof}

 \end{document}